\documentclass[11pt]{article}
\usepackage[utf8]{inputenc}
\usepackage{amssymb,graphicx,mathtools,amsfonts}
\usepackage{amsmath}
\usepackage{amsthm}
\usepackage{hyperref}
\usepackage[dvipsnames]{xcolor}
\usepackage{dsfont}
\usepackage{comment}

\usepackage{a4wide}
\usepackage{esint}
\catcode`\@=11
\def\theequation{\@arabic{\c@section}.\@arabic{\c@equation}}
\catcode`\@=12
\usepackage{enumitem}
\graphicspath{{images/}}
\newtheorem{theorem}{Theorem}[section] 
\newtheorem{lemma}{Lemma}[section] 
\newtheorem{definition}{Definition}[section] 
\newtheorem{proposition}{Proposition}[section]

\newtheorem*{remark}{Remark}

\allowdisplaybreaks 

\newcommand{\al} {\alpha}

\newcommand{\om} {\Omega}

\newcommand{\la} {\lambda}

\newcommand{\real}{\mathbb{R}}

\newcommand{\rnn}{\mathbb{R}^{N}}

\newcommand{\lv}{\lVert}
\newcommand{\rv}{\rVert}

\newcommand{\grad}{\nabla}

\newcommand{\Gr}{\Delta_{\gamma}}
\newcommand{\gradgr}{\nabla_{\gamma}}
\newcommand{\spacegr}{H^{1}_{0,\gamma}(\Omega)}
\newcommand{\spaceg}{H^{1}_{\gamma}(\Omega)}
\newcommand{\tnorm}[1]{\left|\!\left|\!\left| #1 \right|\!\right|\!\right|}

\title{Existence and multiplicity of solutions for a critical Grushin problem with a singular nonlinearity}

\author{Shammi Malhotra\footnote{Department of Mathematics, Indian Institute of Technology Delhi, Hauz Khas New Delhi 110016,  India, shammi22malhotra@gmail.com}}

\date{}

\begin{document}
\maketitle

\begin{abstract}
\noindent We investigate the existence and multiplicity of positive solutions to the problem
\begin{equation*}
\begin{cases}
\begin{aligned}
    - \Delta_{\gamma} u &= \lambda u^{p} + u^{-\delta} &\quad \text{in } \Omega,\\
    u &= 0 &\quad \text{on } \partial \Omega,
\end{aligned}
\end{cases}
\end{equation*}
where $\Delta_{\gamma}$ denotes the Grushin operator defined by
\begin{equation*}
\Delta_{\gamma} := \Delta_x + (1+\gamma)^2 |x|^{2\gamma}\Delta_y,
\end{equation*}
with $\gamma>0$, $z=(x,y)\in \rnn$, $N=n+m$, $n\geq 1$, $m\geq 1$, $\Omega \subset \rnn$ a smooth bounded domain, $\lambda>0$, $1<p<\infty$, and $\delta>0$.

\noindent The analysis depends on the exponent $p$, which may be subcritical, critical, or supercritical, that is, $p<2_\gamma^*-1$, $p=2_\gamma^*-1$, or $p>2_\gamma^*-1$, respectively, where $2_\gamma^*=\frac{2Q}{Q-2}$ is the critical Sobolev exponent associated with the Grushin operator, and $Q=m+(1+\gamma)n$ is the corresponding homogeneous dimension.


\end{abstract}

\section{Introduction}

In this paper, we study a critical exponent problem associated with the Grushin operator. The origin of this operator can be traced back to Tricomi \cite{tricomi_grushin}, who considered the prototype
\begin{equation*}
    \frac{\partial^2}{\partial x^2} + x \frac{\partial^2}{\partial y^2} \quad \text{in } \real^2.
\end{equation*}
This was later generalized by Baouendi \cite{baouendi_phd_thesis} and subsequently by Grushin \cite{grushin_origin}, who introduced the operator in the form
\begin{equation*}
    \Gr := \Delta_x + (1+\gamma)^2 |x|^{2\gamma} \Delta_y,
\end{equation*}
where $\gamma > 0$, $z = (x,y) \in \rnn$, $N = n + m$, with $n \geq 1$ and $m \geq 1$.

In \cite{grushin_origin}, Grushin established the hypoellipticity of $\Delta_{\gamma}$ for $\gamma \in \mathbb{N}$, showing that if $f \in C^{\infty}(\Omega)$, then any distributional solution $u$ of $\Gr u = f$ in $\Omega$ also belongs to $C^{\infty}(\Omega)$. When $\gamma = 0$, the operator $\Gr$ reduces to the classical Laplacian. However, unlike the Laplacian, $\Gr$ is not uniformly elliptic due to the presence of the weight $|x|^{2\gamma}$, which induces degeneracy along the set
\begin{equation*}
\Sigma := \{0\} \times \real^m \subset \real^n \times \real^m.
\end{equation*}

Since $\Gr$ is uniformly elliptic away from $\Sigma$, we assume that $\Omega \cap \Sigma \neq \varnothing$. Furthermore, owing to the translation invariance of $\Gr$ in the $y$-variable, we may assume without loss of generality that $0 \in \Omega$.

The degeneracy set $\Sigma$ plays a crucial role in determining the qualitative behavior of solutions. This influence is clearly illustrated in \cite{ferrero_paolo_asymptotics}, where the authors established asymptotic properties and unique continuation results for solutions near points in $\Sigma \cap \Omega$. In particular, the behavior of solutions in the vicinity of the degeneracy set differs significantly from that of solutions to the classical Laplacian. Their approach relies on an Almgren-type monotonicity formula.

Beginning in the 1980s, Franchi and Lanconelli, through a series of works \cite{franchi_lanconelli_origin_lambda_operators,franchi_lanconelli_lambda_op_holder_cty,franchi_lanconelli_lambda_op_harnack}, introduced and studied a class of degenerate elliptic operators, of which the Grushin operator is a notable example. These operators are of the form
\begin{equation*}
\Delta_{\lambda}
=
\sum_{i=1}^{N}
\partial_{x_i}
\bigl(
\lambda_i^2 \, \partial_{x_i}
\bigr),
\end{equation*}
where the coefficients $\lambda_i$ satisfy suitable admissibility conditions.

Furthermore, the Grushin operator provides a fundamental example of an $X$-elliptic operator in the sense of \cite{lanconelli_kogoj_x_elliptic_harnack}. More precisely, it is uniformly $X$-elliptic with respect to the family of vector fields $X = (X_1, \dots, X_N)$ defined by
\begin{equation*} 
X_i = \frac{\partial}{\partial x_i}, \quad i = 1, \dots, n; 
\qquad
X_{n+j} = |x|^{\gamma}\,\frac{\partial}{\partial y_j}, \quad j = 1, \dots, m.
\end{equation*}
With this notation, the Grushin operator can be written in the sum-of-squares form
\begin{equation*}
\Delta_{\gamma} = \sum_{j=1}^{N} X_j^2.
\end{equation*}

Moreover, when $\gamma \in \mathbb{N}$, the Grushin operator belongs to the class of H\"{o}rmander operators, as it can be expressed as a sum of squares of smooth vector fields generating a Lie algebra of maximal rank at every point. In particular, the system $X = (X_1, \dots, X_N)$ satisfies H\"{o}rmander’s finite rank condition; see \cite{hormander_rankcondn}.

The Grushin operator exhibits a mixed homogeneity that depends on whether the underlying points lie on the degeneracy set or away from it. On the set $\Sigma$, the associated homogeneous dimension is given by
\begin{equation}\label{eq:homogeneous_dim}
    Q = m + (1+\gamma)n.
\end{equation}
In contrast, away from $\Sigma$, the operator becomes uniformly elliptic, and the homogeneous dimension coincides with the classical Euclidean one, namely $N = n + m$.

To describe the natural geometry induced by the operator, for each $t > 0$ we introduce the anisotropic dilations $\delta_t : \mathbb{R}^N \to \mathbb{R}^N$ associated with $\Delta_\gamma$, defined by
\begin{equation}\label{def:natural_dialations}
    \delta_t(x, y) = (t x, t^{1+\gamma} y), \quad \text{for all } (x, y) \in \mathbb{R}^N = \mathbb{R}^n \times \mathbb{R}^m.
\end{equation}
With respect to these dilations, the operator $\Delta_\gamma$ is homogeneous of degree two. More precisely, one has
\begin{equation*}
    \Gr (u \circ \delta_t) = t^2 (\Gr u) \circ \delta_t.
\end{equation*}
We denote by $T : \mathbb{R}^N \to \mathbb{R}^N$ the infinitesimal generator of the group $\{\delta_t\}_{t>0}$, that is, the vector field
\begin{equation*}
    T(x, y) = (x, (1+\gamma)y), \quad \text{for all } (x, y) \in \mathbb{R}^n \times \mathbb{R}^m.
\end{equation*}

A natural gauge associated with the Grushin operator is defined by
\begin{equation}\label{def:gauge_grushin}
    d(z) = \left( |x|^{2(1+\gamma)} + |y|^2 \right)^{\frac{1}{2(1+\gamma)}},
\end{equation}
where $z = (x, y) \in \mathbb{R}^n \times \mathbb{R}^m$. The function $d$ is homogeneous of degree one with respect to the dilations $\delta_t$ defined in \eqref{def:natural_dialations}. Accordingly, for any $z_0 \in \mathbb{R}^N$ and $r > 0$, we define the $d$-ball by
\begin{equation*}
    B_d(z_0, r) = \{ z \in \mathbb{R}^N : d(z - z_0) < r \}.
\end{equation*}
It follows directly from the homogeneity of $d$ that
\begin{equation*}
    \delta_t\bigl(B_d(z_0, r)\bigr) = B_d\bigl(\delta_t(z_0), t r\bigr).
\end{equation*}

In fact, the gauge function introduced in \eqref{def:gauge_grushin} provides the fundamental solution of the Grushin operator. More precisely, there exists a constant $C > 0$, depending only on $\gamma$ and $Q$, such that
\begin{equation*}
    \Gamma(z) = \frac{C}{d(z)^{Q-2}}
\end{equation*}
is the fundamental solution of $-\Delta_\gamma$ with pole at the origin.

The Grushin operator also admits a representation in polar-type coordinates adapted to its intrinsic geometry; see \cite{garofalo_grushin_polar}. Specifically, one has
\begin{equation*}
\Gr = \vartheta_{\gamma} \left( \frac{\partial^2}{\partial \rho^2} + \frac{Q-1}{\rho} \frac{\partial}{\partial \rho} + \left(\frac{\gamma + 1}{\rho}\right)^2 \mathcal{L}_{\theta} \right),
\end{equation*}
where $\rho = d(x,y)$, with $d$ as in \eqref{def:gauge_grushin}, and $\theta \in \mathbb{S}^{N-1} := \{ (x,y) \in \mathbb{R}^N : d(x,y) = 1 \}$. Moreover, $\vartheta_{\gamma}(x,y) = \frac{|x|^{2\gamma}}{d(x,y)^{2\gamma}}$ for $(x,y) \in \mathbb{R}^N \setminus \{0\}$, and $\mathcal{L}_{\theta}$ denotes the analogue of the Laplace--Beltrami operator acting on the Grushin unit sphere $\mathbb{S}^{N-1}$.

The homogeneous dimension $Q$ defined in \eqref{eq:homogeneous_dim} plays a fundamental role in the associated Sobolev embedding; see \cite{monti_morbidelli_grushin}. In particular, one has
\begin{equation*}
    \left( \int_{\mathbb{R}^N} |u|^{2_{\gamma}^*}\, dz \right)^{\frac{1}{2_{\gamma}^*}} 
    \leq C \left( \int_{\mathbb{R}^N} |\nabla_{\gamma} u|^2 \, dz \right)^{\frac{1}{2}},
\end{equation*}
for all $u \in C_c^{\infty}(\Omega)$, where $2_{\gamma}^* = \frac{2Q}{Q-2}$ denotes the critical Sobolev exponent.

This naturally leads to the study of critical exponent problems associated with the Grushin operator. Such problems trace their origin to the seminal work of Brezis and Nirenberg \cite{brezis_1983} for the Laplacian on bounded domains. In the Grushin setting, a systematic investigation was initiated by Monti and Morbidelli \cite{monti_morbidelli_grushin}, who studied the critical problem
\begin{equation}\label{eq:critical_grushin}
    - \Gr u = u^{2_{\gamma}^* - 1} \quad \text{in } \mathbb{R}^N.
\end{equation}
They established existence results and proved symmetry properties of solutions via a Kelvin transform adapted to the Grushin structure. Subsequently, Loiudice \cite{loiudice_asymtotic_behaviour} analyzed the asymptotic behavior of solutions to \eqref{eq:critical_grushin}.

These developments have stimulated extensive research on critical exponent problems in various settings; see, for instance, \cite{sanjit_bal2024multiplicity, bisci_malachini_secchi_bifurcation_multiplicity, zhang_yang_grushin_hardy, zhang_yang_grushin_hardy_asymptotic} and the references therein. In particular, an analogue of the Brezis--Nirenberg theorem for the Grushin operator was established in \cite{alves_tyagi_brezis-nirenberg}, where also a result of Berestycki--Lions type was obtained.

One important class of perturbations of critical exponent problems involves singular nonlinearities. Before addressing the corresponding critical problems, it is instructive to recall the purely singular case. A seminal contribution in this direction is due to Crandall \textit{et al.} \cite{crandaal}, who studied the model problem
\begin{equation}\label{eq:laplace_purely_singular}
   \left\{
	\begin{aligned}
		-\Delta u &= u^{-\delta}, \quad u>0 && \text{in } \Omega,\\
		u &= 0 && \text{on } \partial\Omega,
	\end{aligned}
    \right.
\end{equation}
for $\delta > 0$. The presence of the singular term has a strong impact on the regularity and integrability properties of solutions. In particular, it was shown in \cite{mckenna_integrability_singular} that the solution $u$ to \eqref{eq:laplace_purely_singular} belongs to $H_0^1(\Omega)$ if and only if $\delta \in (0,3)$, while $u \in C^1(\overline{\Omega})$ when $\delta \in (0,1)$. 

The study of critical problems with singular perturbations was initiated by Hirano \textit{et al.}, who established the existence of at least two solutions for $\delta \in (0,1)$ via the fibering map method on the Nehari manifold. Independently, Haito \cite{Haito} obtained analogous results using the sub--supersolution method. These results were later extended to the range $\delta \in (0,3)$ in \cite{adimurthi_giacomoni_strong_singular}. However, for $\delta \geq 3$, the singular term is no longer integrable, and classical variational techniques are no longer applicable. To overcome this difficulty, Hirano \textit{et al.} \cite{hirano_fully_singular} employed tools from nonsmooth analysis (see also \cite{corvellec_nonsmooth_analysis, jabri_nonsmooth_book} and the appendix of \cite{hirano_fully_singular}) together with a suitable translation of the problem, thereby establishing existence results for all $\delta > 0$ in the Brezis--Nirenberg framework.

Motivated by these developments, we investigate analogous problems in the Grushin setting. For the case $\delta < 1$, such problems have already been studied in \cite{sanjit_bal2024multiplicity}. More recently, in \cite{baldelli2025existence}, the authors considered a critical exponent problem for the Grushin operator involving both a singular term and a convective term in the whole space $\mathbb{R}^N$. Their approach relies on truncating the singular term, freezing the convective term, and applying variational methods, leading also to asymptotic estimates in the absence of convection.

In this work, we consider the following problem
\begin{equation}\label{eq:main_eq_grushin}\tag{$\mathcal{P}_{\lambda}$}
\begin{cases}
\begin{aligned}
    - \Delta_{\gamma} u &= \lambda  u^{2_\gamma^* - 1} + u^{-\delta}  && \text{in } \Omega,\\
    u &> 0 && \text{in } \Omega,\\
    u &= 0 && \text{on } \partial \Omega,
\end{aligned}
\end{cases}
\end{equation}
for all $\delta > 0$.

Since our aim is to treat the full range $\delta > 0$, we adopt the nonsmooth analysis approach developed in \cite{hirano_fully_singular}. We begin by studying a generalized problem for which existence is ensured under the presence of suitable subsolutions and supersolutions. This result serves as a key tool in the analysis of \eqref{eq:main_eq_grushin}. We then construct appropriate subsolutions and supersolutions, which naturally leads to the study of the associated purely singular problem without the critical term. Following \cite{canino_degiovanni_fully_singular}, we introduce a sequence of approximate problems and solve them. By employing $\Gamma$-convergence techniques (see \cite{dal_maso_gamma_convergence}), we pass to the limit and obtain a solution.

We subsequently turn to a Brezis--Nirenberg type problem associated with the Grushin operator, namely
\begin{equation}\label{eq:singular_semilinear_problem}\tag{$\mathcal{S}$}
\begin{cases}
-\Gr u = u^{-\delta} + \lambda u^{p} & \text{in } \Omega, \\
u = 0 & \text{on } \partial\Omega,
\end{cases}
\end{equation}
where $\Omega$ is a bounded domain, $\delta > 0$, $\lambda > 0$, and $p \geq 1$. The perturbation term $u^p$ may be subcritical, critical, or supercritical depending on the value of $p$, but it can be handled in a unified manner. To address the integrability issues, we consider a translated version of \eqref{eq:singular_semilinear_problem}. Combining the general framework developed earlier with blow-up analysis in the critical case, we establish the existence of solutions.

Finally, in the presence of both concave and convex nonlinearities, one expects multiplicity phenomena, as observed in the classical work of Ambrosetti, Brezis, and Cerami \cite{abc_laplacian}. In this spirit, we conclude by proving the existence of a second solution via a suitable linking argument, stated later in Section \ref{sec:preliminaries}.

Our main result is the following 

\begin{theorem}
Let $\Omega$ be a bounded domain in $\mathbb{R}^{N}$, let $\delta>0$, and let $p>1$. Then the following assertions hold.

\medskip

\noindent\textnormal{(I) Subcritical and critical cases} $(1< p \leq 2_{\gamma}^*)$.
There exists $\Lambda>0$ such that:
\begin{enumerate}
  \item[(1)] For every $\lambda \in (0,\Lambda)$, there exist at least two positive weak solutions of \eqref{eq:singular_semilinear_problem}$_{\lambda}$ belonging to $C_{loc}^{0,\alpha}(\Omega) \cap L^{\infty}(\Omega)$. Moreover, one of them, denoted by $u_{\lambda}$, satisfies:
  \begin{enumerate}
    \item[(i)] $w > u_{\lambda}$ in $\Omega$ for any positive weak solution $w \neq u_{\lambda}$ of \eqref{eq:singular_semilinear_problem}$_{\lambda}$;
    \item[(ii)] $u_{\lambda}$ is strictly increasing with respect to $\lambda \in (0,\Lambda)$, that is,
    \[
      u_{\mu} > u_{\lambda} \quad \text{in } \Omega \quad \text{for } \mu \in (\lambda,\Lambda).
    \]
  \end{enumerate}
  \item[(2)] For $\lambda = \Lambda$, there exists at least one positive weak solution of \eqref{eq:singular_semilinear_problem}$_{\Lambda}$ belonging to $C_{loc}^{0,\alpha}(\Omega) \cap L^{\infty}(\Omega)$.
  \item[(3)] For any $\lambda > \Lambda$, there exists no positive weak solution of \eqref{eq:singular_semilinear_problem}$_{\lambda}$.
\end{enumerate}

\medskip

\noindent\textnormal{(II) Supercritical case} $(p > 2_{\gamma}^*).$
There exists $\Lambda>0$ such that:
\begin{enumerate}
 \item[(1)] For every $\lambda \in (0,\Lambda)$, there exists at least one positive weak solution $u_{\lambda} \in C_{loc}^{0,\alpha}(\Omega) \cap L^{\infty}(\Omega)$ of \eqref{eq:singular_semilinear_problem}$_{\lambda}$ satisfying:
  \begin{enumerate}
    \item[(i)] $w > u_{\lambda}$ in $\Omega$ for any positive weak solution $w \neq u_{\lambda}$ of \eqref{eq:singular_semilinear_problem}$_{\lambda}$ belonging to $L^{\infty}(\Omega)$;
    \item[(ii)] $u_{\lambda}$ is strictly increasing with respect to $\lambda \in (0,\Lambda)$, that is,
    \[
      u_{\mu} > u_{\lambda} \quad \text{in } \Omega \quad \text{for } \mu \in (\lambda,\Lambda).
    \]
  \end{enumerate}
  \item[(2)] For $\lambda = \Lambda$, there exists a positive weak solution of \eqref{eq:singular_semilinear_problem}$_{\Lambda}$ belonging to $L^{p+1}(\Omega)$.
  \item[(3)] For any $\lambda > \Lambda$, there exists no positive weak solution of \eqref{eq:singular_semilinear_problem}$_{\lambda}$ belonging to $L^{\infty}(\Omega)$.
\end{enumerate}
\end{theorem}

A further difficulty in our analysis arises from the degeneracy of the Grushin operator, which hinders both the positivity and regularity of solutions. To overcome this issue, we first establish a Strong Maximum Principle (see Lemma \ref{lem:SMP_grushin}) by deriving a suitable logarithmic estimate. This is followed by a uniform estimate result (Lemma \ref{lem:uniform_estimate_grushin}), which plays a crucial role in controlling the behavior of solutions. In addition, we make use of a regularity result from Theorem 5.5 in \cite{lanconelli_holder_Xelliptic}.

The paper is organized as follows. In Section \ref{sec:preliminaries}, we introduce the basic notation, recall essential tools from nonsmooth analysis, and establish the Strong Maximum Principle. In Section \ref{sec:brezis_nirenberg_general}, we study the Brezis-Nirenberg problem for the Grushin operator in a general framework and prove results of Perron-type. Section \ref{sec:purely_singular} is devoted to the analysis of the purely singular problem, where we establish the existence of solutions. Finally, in Section \ref{sec:brezis_nirenberg_grushin}, we address the Brezis-Nirenberg problem in the Grushin setting, proving the existence of a positive solution as well as a multiplicity result via a linking theorem.
o'

\section{Preliminaries} \label{sec:preliminaries}

We start this section by introducing the following notation, which will be used throughout the paper:
\begin{enumerate}
    \item[$(i)$] $C$ denotes a generic positive constant, whose value may change from one line to another.
    
    \item[$(ii)$] For any measurable set $A \subset \rnn$, the symbol $|A|$ stands for its $N$-dimensional Lebesgue measure.
     
    \item[$(iv)$] For $1 \leq p < \infty$, the norm in $L^p(\om)$ is given by $\lv u \rv_{L^p(\om)}^p = \lv u \rv_{p}^p := \int_{\om} |u|^p\,dz.$
    
    \item[$(v)$] We define $C_c^\infty(\om):=\{u \in C^\infty(\om) : \operatorname{supp}(u)\Subset \om\}$.
   
\end{enumerate}

Next, we introduce the spaces to study problems associated with the Grushin operator. 
For an open bounded subset $\Omega \subseteq \mathbb{R}^N$, we have 
\begin{equation*}
    H_\gamma^1(\Omega) = \left\{ u \in L^2(\Omega) : |\nabla_\gamma u| \in L^2(\Omega) \right\}
\end{equation*}
endowed with the inner product $\langle \gradgr u, \gradgr v\rangle:= \int_{\om} [\grad_x u \cdot \grad_x v + (1+\gamma)^2 |x|^{2\gamma} \grad_y u \cdot \grad_y v ]\,dz$. Hence, the associated norm is denoted and defined as
\begin{equation*}
    \lv  u \rv_{\gamma}^2 := \int_{\om} | \gradgr u (z) |^2 \, dz = \int_{\om} [ |\grad_x u|^2 + (1 + \gamma )^2 |x|^{2 \gamma} |\grad_y u|^2] \, dz.
\end{equation*}

The space $\spacegr$ is defined as the completion of $C_c^\infty(\Omega)$ with respect to the norm 
\begin{equation*}
    \lv \cdot \rv^2 := \int_{\om} | \gradgr  \cdot |^2 dz.
\end{equation*}
Moreover, by Proposition $2.1$ of \cite{alves_tyagi_brezis-nirenberg}, we have the embedding
\begin{equation*}
   \spacegr \hookrightarrow L^q(\om),
\end{equation*}
for any $q \in [2, 2_\gamma^*]$ and  $H_\gamma^1(\Omega) \hookrightarrow\hookrightarrow L^q(\Omega)$, for every $q \in [1, 2_\gamma^*)$. This mainly comes from the Sobolev embedding inequality, given by
\begin{equation}\label{ineq:Sobolev_embedding_grushin}
    \left( \int_{\om} |u|^{2_{\gamma}^*} \, dz \right)^{2/2_{\gamma}^*} \leq C \int_{\om} |\gradgr u |^2 dz, \quad \text{for all } u \in C_{c}^{\infty}(\real). 
\end{equation}
One can define the best constant in the previous inequality \eqref{ineq:Sobolev_embedding_grushin}, known as Sobolev constant, as
\begin{equation}\label{best_constant_grushin}
    S_{\gamma} := \inf_{\substack{u \in D_0^\gamma(\om) \\ u \neq 0}} \frac{\int_{\mathbb{R}^N} |\nabla_\gamma u|^2 \, \mathrm{d}z}{\left( \int_{\mathbb{R}^N} |u|^{2_\gamma^*} \, \mathrm{d}z \right)^{2/2_\gamma^*}},
\end{equation}
where $D_0^\gamma(\om)$ is the closure of the space $C_c^{\infty}(\om)$ with respect to the norm $\lv . \rv$. It is independent of domain $\om$ and achieved only when $\om = \rnn$. Moreover, we denote by $H_0^{-1,\gamma}(\om)$ the dual space of $\spacegr$.\\


Coming back to the optimal constant defined in \eqref{best_constant_grushin}, its minimizers satisfies the following equation 
\begin{equation}\label{eq:critical_exponent_grushin}
-\Delta_\gamma v = v^{2_{\gamma}^*-1} \quad \text{ in } \rnn,
\end{equation}
and they are invariant under the translations in the $y$ variable and the rescaling defined in \eqref{def:natural_dialations}. Indeed, let $\rho > 0$ and define
\begin{equation*}
    v^{e,\rho}(z) := \rho^{\frac{Q-2}{2}} v(\rho x, \rho^{1+\gamma} y + e), 
\end{equation*}
where $z = (x, y) \in \mathbb{R}^{m+n}$ and $e \in \mathbb{R}^n$. If $v$ satisfies \eqref{eq:critical_exponent_grushin}, then $v^{e,\rho}$ also satisfies \eqref{eq:critical_exponent_grushin}. Also, it is easy to verify the following invariances
\begin{equation*}
    \|\nabla_\gamma v^{e,\rho}\|_{L^2(\mathbb{R}^N)} = \|\nabla_\gamma v\|_{L^2(\mathbb{R}^N)} \text{ and } \|v^{e,\rho}\|_{L^{2_\gamma^*}(\mathbb{R}^N)} = \|v\|_{L^{2_\gamma^*}(\mathbb{R}^N)}.
\end{equation*}


Moreover, the asymptotic estimate of $v$ is also known and given by 
\begin{equation}\label{eq:asymptotic_estimate_fundamental}
    \frac{C_1}{d(z)^{Q-2}} \leq v(z) \leq \frac{C_2}{d(z)^{Q-2}}.
\end{equation}
  For $\varepsilon > 0$, take $v_{\varepsilon} = v^{0, \varepsilon^{-1}}$ and consider the $d$-ball $B_{2R}(z_0) \Subset \Omega \setminus \Sigma$. Choose a function $\zeta \in C_c^\infty(B_{2R}(z_0))$ such that $0 \le \zeta \le 1$ in $B_{2R}(z_0)$ and $\zeta = 1$ in $B_R(z_0)$. Now, define
\begin{equation}\label{eq:grushin_bubbles}
    u_\varepsilon(z) = \zeta(z)v_\varepsilon(z). 
\end{equation}
Then we have the following blow-up estimates borrowed from \cite[Proposition $4.2$]{alves_tyagi_brezis-nirenberg} and \cite[Proposition $0.4$]{sanjit_bal2024multiplicity} which are crucial for our analysis. 

\begin{lemma}\label{lem:blow_up_grushin}
As $\varepsilon \to 0$, the following conclusions hold:
\begin{itemize}
    \item[(i)] $\displaystyle \int_\Omega |\grad_\gamma u_\varepsilon|^2 \, dz \le \mathcal{S}_\lambda^{\frac{Q}{2}} + O(\varepsilon^{Q-2})$,
    \item[(ii)] $\displaystyle \int_\Omega |u_\varepsilon|^{2_\gamma^*} \, dz = \mathcal{S}_\lambda^{\frac{Q}{2}} + O(\varepsilon^Q)$,
    \item[(iii)] $\displaystyle \int_\Omega |u_\varepsilon|^2 dz \ge \begin{cases} 
        C\varepsilon^2 + O(\varepsilon^{Q-2}) \text{ if } Q > 4 \\
        C\varepsilon^2 |\ln \varepsilon| + O(\varepsilon^2) \text{ if } Q = 4 \\
        C\varepsilon^{Q-2} + O(\varepsilon^2) \text{ if } Q < 4.
    \end{cases}$
    \item[(iv)] For $\frac{2_\gamma^*}{2} < \lambda < 2_\gamma^*$, we have
    \[
        \int_\Omega |u_\varepsilon|^\lambda \, dz = O(\varepsilon^{Q - \frac{\lambda(Q-2)}{2}}).
    \]
\end{itemize}
\end{lemma}

Next, we have the polar coordinates formula for $d$-radial functions. 
\begin{proposition}
Let $0 \le r_1 < r_2$ and let $P : [r_1, r_2] \to \mathbb{R}$ be any measurable function. Then the following identity holds:
\begin{equation*}
\int_{B_{r_2}(0) \setminus B_{r_1}(0)} P\big(d(z)\big)\, dz
=
Q \, |B_1(0)| \int_{r_1}^{r_2} P(\rho)\, \rho^{Q - 1}\, d\rho .
\end{equation*}   
\end{proposition}

\subsection{Nonsmooth Analysis}
As we are interested in problem with a singular nonlinearity, we need some tools of non-smooth analysis to deal with it. First we introduce the Fr\'{e}chet subdifferential point for a non-smooth functional that will work as the notion of the critical point.  

Let $H$ be a Hilbert space and let $I : H \to (-\infty, \infty]$ be a functional with its effective domain defined as 
\[
\mathcal{D}(I) := \{ u \in H : I(u) < \infty \}.
\]
If $I \not\equiv \infty$, we call it to be proper. Now, let $I : H \to (-\infty, \infty]$ be a proper, lower semicontinuous functional.
Then for every $u \in \mathcal{D}(I)$, the Fr\'{e}chet subdifferential of $I$ at $u$ is given by
\begin{equation*}
\partial^{-} I(u)
=
\left\{
\alpha \in H :
\lim_{v \to u}
\frac{ I(v) - I(u) - \langle \alpha, v - u \rangle }
{\|v - u\|}
\ge 0
\right\}.
\end{equation*}
The subdifferential satisfies the following properties. 
\begin{enumerate}
    \item $\partial^{-} I(u)$ may be empty; however, it is always closed and convex.
    \item If $u \in \mathcal{D}(I)$ is a local minimizer of $I$, then $0 \in \partial^{-} I(u)$.
    \item If $I_0 : H \to (-\infty, \infty]$ is a proper, lower semicontinuous, convex functional, $I_1 : H \to \mathbb{R}$ is of class $C^1$, and $I = I_0 + I_1$, then the Fréchet subdifferential satisfies
\begin{equation*}
\partial^{-} I(u)
=
\partial I_0(u) + \gradgr I_1(u)
\qquad
\text{for every } u \in \mathcal{D}(I) = \mathcal{D}(I_0),
\end{equation*}
where $\partial I_0$ denotes the usual subdifferential of the convex functional
$I_0$.
\end{enumerate}

Finally, to work with the set $\partial^{-} I$, for every $u \in H$, we define the quantity $\tnorm{\partial^{-} I(u)}$ by
\begin{equation*}
\tnorm{\partial^{-} I(u}
=
\begin{cases}
\displaystyle
\min \{ \|\alpha\| : \alpha \in \partial^{-} I(u) \},
& \text{if } \partial^{-} I(u) \neq \varnothing, \\[1ex]
\infty,
& \text{if } \partial^{-} I(u) = \varnothing .
\end{cases}
\end{equation*}

For functionals with singular nonlinearities, boundedness of the minimizing sequence is not easy to get. To proceed, we require the following definition
\begin{definition}[Cerami Condition]
Let $I : H \to (-\infty, \infty]$ be a proper, lower semicontinuous functional. We say that $I$ satisfies the Cerami condition at level $c$,i.e., $I$ satisfies $(\mathrm{CPS})_c$, if every sequence $\{u_k\} \subset \mathcal{D}(I)$ satisfying
\begin{equation*}
I(u_k) \to c
\qquad \text{and} \qquad
\left( 1 + \|u_k\| \right)
\tnorm{\partial^- I(u_k)} \to 0
\end{equation*}
admits a convergent subsequence in $H$.    
\end{definition}

To obtain a multiplicity result, we apply the following linking theorem proved in Appendix A of \cite{hirano_fully_singular}. 
 
\begin{theorem}\label{thm:Linking_theorem}
Let $H$ be a Hilbert space. Let
$I_0 : H \to (-\infty, \infty]$ be a proper, lower semicontinuous, convex
functional, let
$I_1 : H \to \mathbb{R}$ be of class $C^1$, and define
$I = I_0 + I_1$.
Let $\mathbb{D}^N$ and $\mathbb{S}^{N-1}$ denote respectively the closed unit ball and its boundary sphere in $\mathbb{R}^N$. Assume that $\psi : \mathbb{S}^{N-1} \to \mathcal{D}(I)$ is a continuous mapping such that
\begin{equation*}
\Phi :=
\left\{
\varphi \in C(\mathbb{D}^N, \mathcal{D}(I)) :
\varphi|_{\mathbb{S}^{N-1}} = \psi
\right\}
\neq \varnothing .
\end{equation*}
Let $A$ be a relatively closed subset of $\mathcal{D}(I)$ such that it satisfies the following
\begin{itemize}
    \item $A \cap \psi(\mathbb{S}^{N-1}) = \varnothing$,
    \item $A \cap \varphi(\mathbb{D}^N) \neq \varnothing$ for all $\varphi \in \Phi$,
    \item $\inf I(A) \ge \sup I\bigl( \psi(\mathbb{S}^{N-1}) \bigr)$.
\end{itemize}
Assume furthermore that
\begin{equation*}
c :=
\inf_{\varphi \in \Phi}
\;
\sup_{x \in \mathbb{D}^N}
I(\varphi(x))
\in \mathbb{R},
\end{equation*}
and that $I$ satisfies the condition $(\mathrm{CPS})_c$.
Then there exists $u \in \mathcal{D}(I)$ such that
\begin{equation*}
I(u) = c
\qquad \text{and} \qquad
0 \in \partial^- I(u).
\end{equation*}
Moreover, if $\inf I(A) = c$, then $u \in A \cap \mathcal{D}(I)$.
\end{theorem}

\subsection{Maximum Principles}

We prove the Strong maximum principle for the Grushin operator. But before that we require the following logarithmic lemma.

\begin{lemma}[Logarithmic Lemma]\label{lem:logarithmic_lemma}
Let $u \in \spacegr$ be a nonnegative supersolution of $- \Gr u + h u $ in the sense of distributions where $h \in L_{loc}^{\frac{2Q}{Q+2}}(\om)$ and $d > 0$ be a positive number. Then the following holds
\begin{equation*}
\int_{B_r} |\gradgr \log(u+d)|^2 dz \leq C r^{N-2} + \int_{B_{3r/2}} \frac{h u}{u + d} dz.   
\end{equation*}
If $u \in L_{loc}^{\infty}(\om)$ then it is enough to take $h \in L_{loc}^1(\om)$. 
\end{lemma}
\begin{proof}
Let $\phi \in C_c^\infty(\Omega)$ be such that
\begin{equation*}
\phi \equiv 1 \quad \text{in } B_r, 
\qquad 
|\gradgr \phi| \le \frac{C}{r},
\qquad 
\operatorname{supp}(\phi) \subset B_{3r/2},
\end{equation*}
where $r>0$ is chosen such that $B_{3r/2} \subset \om$. 
Since $u \ge 0$, there exists a nonnegative sequence $\{u_k\}_{k\ge1}$ such that
\begin{equation*}
u_k \in C_c^{\infty}(\om),
\qquad
\operatorname{supp}(u_k) \text{ is compact in } \Omega,
\qquad
\|u_k - u\| \to 0
\quad \text{as } k \to \infty.
\end{equation*}
Now, choosing the test function $\eta_k = (u_k + d)^{-1}\phi^2$, we obtain
\begin{equation}
\label{eq:basic-inequality}
\begin{aligned}
0 \le &
- \int_{B_{3r/2}} 
\gradgr u \cdot \gradgr u_k \,
(u_k + d)^{-2}\phi^2  + \int_{B_{3r/2}} 
\gradgr u \cdot \gradgr \phi \,
\bigl( 2\phi (u_k + d)^{-1} \bigr) \\
& + \int_{B_{3r/2}} h u (u_k +d )^{-1} \phi^2.
\end{aligned}
\end{equation}
For any measurable set $E \subset \Omega$, we estimate
\begin{equation*}
\begin{aligned}
\left|
\int_E 
\gradgr u \cdot \gradgr u_k \,
\phi^2 (u_k + d)^{-2}
\right|
& \le \frac{1}{d^2}
\int_E 
|\gradgr u|\,|\gradgr u_k|
\\
& \le \frac{1}{d^2}
\left( \int_E |\gradgr u|^2 \right)^{1/2}
\left( \int_E |\gradgr u_k|^2 \right)^{1/2}
\\
& \le \frac{1}{d^2} M
\left( \int_E |\gradgr u|^2 \right)^{1/2}.
\end{aligned}
\end{equation*}
Therefore, by the Vitali's Convergence Theorem, we conclude that
\begin{equation}
\label{eq:vitali-limit}
\int_{\Omega}
\gradgr u \cdot \gradgr u_k \,
\phi^2 (u_k + d)^{-2} \, dz
\to
\int_{\Omega}
|\gradgr u|^2
\phi^2 (u + d)^{-2} \, dz,
\quad
\text{as } k \to \infty.
\end{equation}
Using \eqref{eq:vitali-limit} and the Dominated Convergence Theorem in \eqref{eq:basic-inequality} together with Young's inequality, we obtain
\begin{equation*}
\begin{aligned}
0 \le {} &
- \int_{B_{3r/2}}
\left| \gradgr \log (u + d) \right|^2
\phi^2
+ 2
\int_{B_{3r/2}}
\gradgr u \cdot \gradgr \phi \,
( \phi (u + d)^{-1}) + \int_{B_{3r/2}} h u (u +d )^{-1} \phi^2
\\
\le {} &
- \frac{1}{2}
\int_{B_{3r/2}}
\left| \gradgr \log (u + d) \right|^2
\phi^2
+ C
\int_{B_{3r/2}}
|\gradgr \phi|^2 
+ \int_{B_{3r/2}} h u (u +d )^{-1}.
\end{aligned}
\end{equation*}

This completes the proof.
\end{proof}

\begin{lemma}[Strong Maximum Principle]\label{lem:SMP_grushin}
Let $u$ be a non-negative supersolution of $- \Gr u + h u  $ where $h \in L_{loc}^{\frac{2Q}{Q+2}}(\om)$. Then either $u > 0$ in $\om$ or $u \equiv 0$ in $\om$.  
\end{lemma}
\begin{proof}
If $u \equiv 0$, there is nothing to prove. Let $u \not \equiv 0$. Then for any ball $B_r(z_0) \subset B_{3r/2}(z_0) \subset \om $, we define the set $Z : = \{ z \in B_r(z_0): u(z) = 0 \}$. Using the Poincar\'{e} inequality in \cite[Theorem 13.27]{giovanni_fractional} and Lemma \ref{lem:logarithmic_lemma}, we have
\begin{equation*}
\begin{aligned}
 \int_{ B_{r}\left(z_0\right)}\left|\log \left(1+\frac{u}{d}\right)\right|^2 d z & =\int_{ B_{r}\left(z_0\right)}\left|\log \left(1+\frac{u}{d}\right)-m_Z\right|^2 d z \\
 & \leq C \int_{ B_{r}\left(z_0\right)}\left|\gradgr \log \left(1+\frac{u}{d}\right)\right|^2 d z \\
 & \leq C r^{N-2} +\int_{B_{3r/2}(z_0)} h u (u +d )^{-1} dz
\end{aligned}
\end{equation*}
where $m_Z$ is the mean of $v:=\log (1+u / d) \in \spacegr$ on the set $Z$, that is, $m_Z:=\frac{1}{|Z|} \int_Z \log \left(1+\frac{u}{d}\right) d z=0.$ Let $d \to 0$, we have a contradiction unless $|Z| = 0$. Thus $u > 0$ in $B_r(z_0)$ for any $z_0 \in \om$. 
Now, since for any connected compact set $K \subset \om$, we can cover it with a finite number of balls $B_{r / 2}\left(z_1\right), \ldots, B_{r / 2}\left(z_k\right)$ such that $z_i \in K$, $B_r(z_i) \subset \om$ and
\begin{equation*}
\left|B_{\frac{r}{2}}\left(z_i\right) \cap B_{\frac{r}{2}}\left(z_{i+1}\right)\right|>0, \quad i=1, \ldots, k-1 .
\end{equation*}
We must have $u > 0$ in $K$. But $K$ is arbitrary, thus $u > 0$ in $\om$.
\end{proof}

\begin{remark}
\begin{enumerate}
    \item[$(i)$] If $u \in L_{loc}^{\infty}(\om)$, we can take $h \in L_{loc}^{\infty}(\om)$.
    \item[$(ii)$] If $h u \in L_{loc}^1(\om)$ is nonnegative, then the same conclusion holds in the above lemma. 
\end{enumerate}
\end{remark}

\begin{lemma}\label{lem:uniform_estimate_grushin}
Let $\om \subset \rnn$ be a bounded domain. If $u \in \spacegr$ satisfies $-\Gr u \leq v$ in the sense of distributions where $v \in L^{\alpha}(\om)$ with $\alpha > Q/2$. Then $u \in L^{\infty}(\om)$.
\end{lemma}
\begin{proof}
Let $G_t(\beta) = (\beta - t)^+$ and $u_k$ be a sequence such that $u_k \in C_c^{\infty}(\om)$ and $\lv u_k -u\rv \to 0$ as $k \to \infty$. Then we have 
\begin{equation*}
\int_{\om} \gradgr u \cdot \gradgr G_t(u_k) \, dz \leq \int_{\om} u G_t(u_k) \, dz.
\end{equation*}
Passing to the limit we have 
\begin{equation*}
\int_{\om} \gradgr u \cdot \gradgr G_t(u) \, dz \leq \int_{\om} u G_t(u) \, dz.
\end{equation*}
Now the proof follows with a slight adaptation of the proof of Theorem $6.6$ in \cite{boccardo_book}.
\end{proof}

It is easy to prove the lemma by following along the same lines of Lemma $A.1$ of \cite{hirano_concave_convex}.
\begin{lemma}\label{lem:density_grushin_space}
For every non-negative $u \in \spacegr$ there exists an increasing sequence $\{ u_k \} \subset \spacegr \cap L^{\infty}(\om)$ such that each function $u_k$ has a compact support and converges strongly to $u$.
\end{lemma}

We end this section by stating an important inequality.

\begin{proposition}[\cite{shafrir2000asymptotic}] \label{ineq:imp_ineq}
    If $p \geq 2$, then $$| a - b|^p -|a|^p - p |a|^{p-2} a.b \leq \frac{p(p-1)}{2}(|a| + |b|)^{p-2}|b|^2, \mbox{ for all } a,b \in \rnn.$$
    If $p \in (1,2)$, then there exists a constant $C_p$, such that
    $$ | a - b|^p -|a|^p - p |a|^{p-2} a.b \leq C_p |b|^p.$$
\end{proposition}



\section{Brezis-Nirenberg problem in a general framework} \label{sec:brezis_nirenberg_general}

In this section, we undertake a rigorous analysis of the following boundary value problem
\begin{equation}
\label{eq:main_problem}\tag{$\mathcal{P}$}
\begin{cases}
\begin{aligned}
- \Gr u + g(z,u) &= f(z,u) \quad &&\text{in } \Omega, \\
u &= 0 \quad &&\text{on } \partial\Omega. 
\end{aligned}
\end{cases}
\end{equation}

Our primary objective is to establish the existence of a solution of \eqref{eq:main_problem} under the assumption that a subsolution or supersolution to the problem is available. This result will serve as a foundational tool for addressing more specialized problems in the subsequent sections.

Before proceeding, we list the standing hypotheses. 

\medskip


\medskip

 \textbf{(A1)} Let $1 \le p \le \frac{Q+2}{Q-2}$. 
Assume that $f : \Omega \times \mathbb{R} \to \mathbb{R}$ is a Carath\'eodory function satisfying the growth condition
\begin{equation*}
|f(z,t)| \le a_1(z) + c_1 |t|^p
\quad \text{for a.e. } z \in \Omega \text{ and for all } t \in \mathbb{R},
\end{equation*}
where $a_1 \in L^{(p+1)/p}(\Omega)$ is a nonnegative function and $\frac{p+1}{p}$ is the conjugate exponent of $p+1$.

\medskip

 \textbf{(A2)} Let $g : \Omega \times \mathbb{R} \to [-\infty,\infty]$ and $G : \Omega \times \mathbb{R} \to (-\infty,\infty]$ be functions such that:

\medskip

 \textbf{(i)} For every $t \in \mathbb{R}$, the mappings $g(\cdot,t) : \Omega \to [-\infty,\infty]$ and $G(\cdot,t) : \Omega \to (-\infty,\infty]$ are measurable.  

\medskip

 \textbf{(ii)} For almost every $z \in \Omega$:
\begin{itemize}
  \item[(a)] The function $G(z,\cdot) : \mathbb{R} \to (-\infty,\infty]$ is lower semicontinuous and convex.
  \item[(b)] The function $G(z,\cdot)$ is differentiable on $\operatorname{Int}\{ v \in \mathbb{R} : G(z,v) < \infty \}$ and satisfies
  \begin{equation*}
\frac{\partial G}{\partial t}(z,t) = g(z,t)
  \quad \text{for every } t \in \operatorname{Int}\{ s \in \mathbb{R} : G(z,s) < \infty \}.
  \end{equation*}
  \item[(c)] The mapping $G(z,\cdot)$ is continuous on $\{ s \in \mathbb{R} : G(z,s) < \infty \}$.
\end{itemize}

\medskip

 \textbf{(iii)} There exist nonnegative functions $a_2\in L^{\frac{2Q}{Q+2}}(\Omega)$ and $a_3 \in L^1(\Omega)$ such that
\begin{equation*}
G(z,t) \ge - a_2(x)\,|t| - a_3(z),
\end{equation*}
for a.e. $z \in \om$ and for all $t \in \real$.

\subsection{General Weak Solution}
To define the notion of the solutions of \eqref{eq:main_problem}, we have the following definitions 

\begin{definition}\label{def:weak_sub_super_solution_grushin}
We say $\varphi$ is a \textit{subsolution} (resp. a \textit{supersolution}) of \eqref{eq:main_problem} if the following hold
\begin{enumerate}
    \item[(i)] $\varphi \in \spaceg$;
    \item[(ii)] $\varphi^+ \in \spacegr$ (resp. $\varphi^- \in \spacegr$);
    \item[(iii)] $g(\cdot, \varphi) \in L_{loc}^1(\Omega)$;
    \item[(iv)] $-\Gr \varphi + g(z, \varphi) - f(z, \varphi) \leqslant 0$ (resp. $\geqslant 0$) in the sense of distributions.
\end{enumerate}
Furthermore, we say $\varphi$ is a \textit{strict subsolution} (resp. a \textit{strict supersolution}) of \eqref{eq:main_problem} if $\varphi$ is a subsolution (resp. a supersolution) of \eqref{eq:main_problem} and
\begin{enumerate}
    \item[(v)] $g^+(\cdot, \varphi) \in L^{\frac{2Q}{Q+2}}(\Omega)$ (resp. $g^-(\cdot, \varphi) \in L^{\frac{2Q}{Q+2}}(\Omega)$);
    \item[(vi)] $\int_\Omega (\gradgr \varphi \gradgr v + g(z, \varphi)v - f(z, \varphi)v) \, dz < 0$ (resp. $> 0$) for all $v \in \spacegr \setminus \{0\}$ with $v \geqslant 0$.
\end{enumerate}
\end{definition}

The above definition leads to the definition of weak solution which is given by

\begin{definition}
We call $\varphi$ to be a \textit{weak solution} of \eqref{eq:main_problem} if $\varphi$ is both a subsolution and a supersolution of \eqref{eq:main_problem}. In particular, if $\varphi$ satisfies 
\begin{enumerate}
    \item $\varphi \in \spacegr$;
    \item $g(\cdot, \varphi) \in L_{loc}^1(\Omega)$;
    \item $-\Gr \varphi + g(z, \varphi) - f(z, \varphi) = 0$.
\end{enumerate}   
\end{definition}

Also, we require a weak sense in which we define the values of functions at the boundary.
\begin{definition}\label{def:boundary_def}
Let $u \in H^{1}_{\gamma, loc}(\Omega)$. We say that $u \le 0$ on $\partial\Omega$ if, for every $\varepsilon > 0$, the function $(u - \varepsilon)^+ \in \spacegr$.
\end{definition}

Next, we recall several properties of the functions $g$ and $G$, which are taken from Lemma $1$ and Lemma $2$ of \cite{hirano_concave_convex}.

\begin{lemma}\label{lem:property_G_1}
Assume that \textnormal{(A1)} and \textnormal{(A2)} hold. 
Let $u$, $v$, and $w$ be real-valued functions defined on $\Omega$ such that $w$ lies in the convex set $\{ \phi : \min \{u, v \} \leq \phi \leq \max \{u, v \} \text{ a.e. in } \om \}$. Then, for almost every $z \in \Omega$, the following estimate holds:
\begin{equation*}
G\bigl(z, w(z)\bigr) \le G^{+}\bigl(z, u(z)\bigr) + G^{+}\bigl(z, v(z)\bigr).
\end{equation*}
\end{lemma}

\begin{lemma}
\label{lem:measurability_properties_G_g}
Assume that \textnormal{(A1)} and \textnormal{(A2)} are satisfied. 
Suppose, in addition, that there exists a measurable function $u_0 : \Omega \to \mathbb{R}$ such that $G\bigl(z, u_0(z)\bigr) \in \mathbb{R}$ for a.e. $z \in \Omega,$
and that the mapping $z \mapsto G\bigl(z, u_0(z)\bigr)$ is measurable. 
Then the following statements hold.
\begin{itemize}
    \item For every measurable function $u : \Omega \to \mathbb{R}$ satisfying
\begin{equation}
\label{eq:G_finite_at_u}
G\bigl(z, u(z)\bigr) \in \mathbb{R}
\quad \text{for almost every } z \in \Omega,
\end{equation}
the function $z \mapsto G\bigl(z, u(z)\bigr)$ is measurable.
    \item If, in addition, the mapping $z \mapsto G\bigl(z, u_0(z)\bigr)$ is measurable, then for every measurable function $u : \Omega \to \mathbb{R}$ satisfying \eqref{eq:G_finite_at_u}, the function $z \mapsto g\bigl(z, u(z)\bigr)$ is measurable. 
    \item For every pair of measurable functions $u, v : \Omega \to \mathbb{R}$ such that $G\bigl(z, u(z)\bigr), \, G\bigl(z, v(z)\bigr) \in \mathbb{R}$ for almost every $z \in \Omega,$ the product
\begin{equation*}
z \mapsto g\bigl(z, u(z)\bigr)\,\bigl(v(z) - u(z)\bigr)
\end{equation*}
is measurable.
\end{itemize}
\end{lemma}

\subsection{Variational Setup}
We now introduce a variational framework for problem~\eqref{eq:main_problem}. 
Let $F : \Omega \times \mathbb{R} \to \mathbb{R}$ be a Carath\'eodory function such that
\begin{equation*}
F(z,0) \in L^{1}(\Omega)
\quad \text{and} \quad
\partial_t F(z,t) = f(z,t)
\quad \text{for } (z,t) \in \Omega \times \mathbb{R}.
\end{equation*}

We define the functional $I : \spacegr \to (-\infty,\infty]$ by
\begin{equation*}
I(u) =
\begin{cases}
\dfrac{1}{2} \displaystyle\int_{\Omega} |\gradgr u|^{2}\,dz
+ \displaystyle\int_{\Omega} G(z,u)\,dz
- \displaystyle\int_{\Omega} F(z,u)\,dz,
& \text{if } G(\cdot,u) \in L^{1}(\Omega), \\[1.2ex]
\infty, & \text{otherwise}.
\end{cases}
\end{equation*}
This definition holds for every $u \in \spacegr$. 
For any subset $K \subset \spacegr$, we also introduce the functional
$I_K : \spacegr \to (-\infty,\infty]$ defined by
\begin{equation}
\label{eq:definition_of_IK}
I_K(u) =
\begin{cases}
I(u), & \text{if } u \in K \text{ and } G(\cdot,u) \in L^{1}(\Omega), \\
\infty, & \text{otherwise}.
\end{cases}
\end{equation}
for every $u \in \spacegr$. \\
The next lemma characterizes the elements lying inside the subdifferential of $I_K$.

\begin{lemma}
\label{lem:variational_inequality_characterization}
Assume \textnormal{(A1)} and \textnormal{(A2)}. 
Let $K$ be a convex subset of $\spacegr$. 
Let $w \in \spacegr$ and let $u \in K$ with $G(\cdot,u) \in L^{1}(\Omega)$. 
Then the following statements are equivalent:
\begin{itemize}
  \item[(i)] $w \in \partial^{-} I_K(u)$;
  \item[(ii)] for every $v \in K$ with $G(\cdot,v) \in L^{1}(\Omega)$, we have  $g(\cdot,u)(v-u) \in L^{1}(\Omega)$ and
  \begin{equation*}
  \begin{aligned}
  \int_{\Omega} \gradgr u \cdot \gradgr (v-u)\, dz
  + \int_{\Omega} g(z,u)\,(v-u)\, dz
  - \int_{\Omega} f(z,u)\,(v-u)\, dz
  \ge \langle w, v-u \rangle .
  \end{aligned}
  \end{equation*}
\end{itemize}
\end{lemma}
\begin{proof}
The result can be established by an argument analogous to that used in the proof of Lemma 3 in \cite{hirano_concave_convex}.
\end{proof}

For any functions $\varphi, \psi : \Omega \to [-\infty,\infty]$, we define the order interval sets $K_{\varphi}^{\psi}$, $K_{\varphi}$, and $K^{\psi}$ by
\begin{equation*}
\begin{aligned}
&K_{\varphi}^{\psi}
= \left\{ u \in \spacegr : \varphi \le u \le \psi \ \text{a.e. in } \Omega \right\},\\
&K_{\varphi} = \left\{ u \in \spacegr : u \geq \varphi \text{ a.e. in } \om \right\}, \quad \text{ and } \quad 
K^{\psi}
= \left\{ u \in \spacegr : u \le \psi \ \text{a.e. in } \Omega \right\}.
\end{aligned}
\end{equation*}

Next we see how we a critical point gives the weak solution of the equation. The proof is essentially the same as of Proposition $3.1$ of \cite{hirano_fully_singular}, however, we provide it for completeness. 

\begin{proposition}
\label{prop:sub_super_solution_criterion}
Assume \textnormal{(A1)} and \textnormal{(A2)}. Suppose, in addition, that at least one of the following conditions is satisfied:
\begin{itemize}
  \item[(i)] $\varphi_1$ is a subsolution of~\eqref{eq:main_problem}, 
  $G(\cdot,v) \in L_{loc}^{1}(\Omega)$ for all $v \in K_{\varphi_1}$, 
  $u \in D(I_{K_{\varphi_1}})$, and $0 \in \partial^{-} I_{K_{\varphi_1}}(u)$;

  \item[(ii)] $\varphi_2$ is a supersolution of~\eqref{eq:main_problem}, 
  $G(\cdot,v) \in L_{loc}^{1}(\Omega)$ for all $v \in K^{\varphi_2}$, 
  $u \in D(I_{K^{\varphi_2}})$, and $0 \in \partial^{-} I_{K^{\varphi_2}}(u)$;

  \item[(iii)] $\varphi_1$ and $\varphi_2$ are, respectively, a subsolution and a supersolution of~\eqref{eq:main_problem}, 
  $\varphi_1 \le \varphi_2$ almost everywhere in $\Omega$, 
  $G(\cdot,\varphi_1),\, G(\cdot,\varphi_2) \in L_{loc}^{1}(\Omega)$, 
  $u \in D(I_{K_{\varphi_1}^{\varphi_2}})$, and $0 \in \partial^{-} I_{K_{\varphi_1}^{\varphi_2}}(u)$.
\end{itemize}
Then $u$ is a weak solution of~\eqref{eq:main_problem}.
\end{proposition}

\begin{proof}
\textnormal{(i)} 
Since $g(\cdot,\varphi_1)$ is measurable and $G(z,u(z)) \in \real$ for almost every $z \in \om$.  Hence, by Lemma~\ref{lem:measurability_properties_G_g}\textnormal{(ii)}, the mapping $g(\cdot,u)$ is measurable.
For each $\psi_0 \in C_c^{\infty}(\Omega)$ with $\psi_0 \ge 0$, we have
\[
G(z, u+\psi_0) - G(z,u) \ge g(z,u)\,\psi_0 \ge g(z,\varphi_1)\,\psi_0,
\]
by the convexity of $G(z,\cdot)$ and increasing nature of $g(z,\cdot)$. This implies $g(\cdot,u)\,\psi_0 \in L^{1}(\Omega)$. By the arbitrariness of $\psi_0$, it follows that $g(\cdot,u) \in L_{loc}^{1}(\Omega)$.

Let $\psi$ be an arbitrary element of $C_c^{\infty}(\Omega)$. Fix $t \in (0,1]$ and define
\begin{equation*}
v_t = \max\{u + t\psi, \varphi_1\}.
\end{equation*}
Then $G(\cdot,v_t) \in L_{loc}^{1}(\Omega)$ and $G(z,v_t) = G(z,u)$ on $\Omega \setminus \operatorname{supp}\psi$. Consequently, $v_t \in D(I_{K_{\varphi_1}})$. Setting $\zeta_t = \bigl(\varphi_1 - (u + t\psi)\bigr)^{+}$, we obtain
\begin{equation}
\label{eq:vt_minus_u_decomposition}
v_t - u = t\psi + \zeta_t .
\end{equation}
Note that $\operatorname{supp} \zeta_t$ is compact and that $|\zeta_t(z)| \le t|\psi(z)|$ for every $z \in \Omega$.
By Lemma~\ref{lem:variational_inequality_characterization}, we have $g(\cdot,u)\,(v_t-u) \in L^{1}(\Omega)$ and
\begin{equation}\label{eq:vt_with_characterisation}
\begin{aligned}
0 &\le \int_{\Omega} \Bigl( \gradgr u \cdot \gradgr (v_t - u) + g(z,u)\,(v_t - u) - f(z,u)\,(v_t - u) \Bigr)\, dz, \\
& =t \int_{\Omega} \Bigl( \gradgr u \cdot \gradgr \psi + g(z,u)\,\psi - f(z,u)\,\psi \Bigr)\, dz  {} + \int_{\Omega} \Bigl( \gradgr u \cdot \gradgr \zeta_t + g(z,u)\, \zeta_t - f(z,u)\, \zeta_t \Bigr)\, dz,
\end{aligned}
\end{equation}
where we used the decomposition \eqref{eq:vt_minus_u_decomposition}.
Since we can find a sequence $\{w_k\} \subset C_c^{\infty}(\Omega)$ such that $w_k \ge 0$, the union
$\bigcup_k \operatorname{supp} w_k$ is contained in a compact subset of $\Omega$, the sequence
$\{\|w_k\|_{\infty}\}$ is bounded, and $\|w_k - \zeta_t\| \to 0$ as $k \to \infty$, using the fact that
$\varphi_1$ is a subsolution together with Lebesgue's convergence theorem, we obtain
\begin{equation}
\label{eq:inequality_with_phi1}
\int_{\Omega} \Bigl( \gradgr \varphi_1 \cdot \gradgr \zeta_t + g(z,\varphi_1)\, \zeta_t - f(z,\varphi_1)\, \zeta_t \Bigr)\, dz \le 0 .
\end{equation}
Subtracting~\eqref{eq:inequality_with_phi1} from~\eqref{eq:vt_with_characterisation} and using the fact that on $Supp(\zeta_t)$, we have $u - \varphi_1 = - (t \psi + \zeta_t)$, we conclude that
\begin{equation*}
\begin{aligned}
0 \le {} &
\int_{\Omega} \Bigl( \gradgr u \cdot \gradgr \psi + g(z,u)\, \psi - f(z,u)\, \psi \Bigr)\, dz
- \frac{1}{t}\int_{\Omega} |\gradgr \zeta_t|^{2}\, dz- \int_{\Omega} \gradgr \psi \cdot \gradgr \zeta_t \, dz \\
& {} + \int_{\Omega} \Biggl( \bigl(g(z,u) - g(z,\varphi_1)\bigr) \frac{\zeta_t}{t}
- \bigl(f(z,u) - f(z,\varphi_1)\bigr) \frac{\zeta_t}{t} \Biggr)\, dz .
\end{aligned}
\end{equation*}
Since $\frac{1}{t}\int_{\Omega} |\gradgr \zeta_t|^{2}\, dz \geq 0$,  $ \int_{\Omega} |\gradgr \zeta_t|^{2}\, dz \to 0$ as $t \to 0^{+}$, and since
$\lvert \zeta_t(z) \rvert / t \le \lvert \psi(z) \rvert$ for every $z \in \Omega$ and $t \in (0,1]$, while
$\operatorname{supp} \psi$ is compact, $\bigl(g(z,u) - g(z,\varphi_1)\bigr)\frac{\zeta_t}{t} \to 0$, and $\bigl(f(z,u) - f(z,\varphi_1)\bigr)\frac{\zeta_t}{t} \to 0$ a.e. in $\om$ as $t\to 0^+$, 
with $g(\cdot,u)$, $g(\cdot,\varphi_1)$, $f(\cdot,u)$, $f(\cdot,\varphi_1) \in L_{loc}^{1}(\Omega)$, the dominated convergence theorem yields
\begin{equation}
\label{eq:limit_variational_inequality}
0 \le \int_{\Omega} \Bigl( \gradgr u \cdot \gradgr \psi + g(z,u)\, \psi - f(z,u)\, \psi \Bigr)\, dz .
\end{equation}
Since $\psi \in C_c^{\infty}(\Omega)$ is arbitrary, it follows that $u$ is a weak solution of~\eqref{eq:main_problem}.

\textnormal{(ii)} By an argument analogous to part~\textnormal{(i)}, we obtain $g(\cdot,u) \in L_{loc}^{1}(\Omega)$. 
Let $\psi \in C_c^{\infty}(\Omega)$ and fix $t \in (0,1]$. Define $v_t = \min\{u + t\psi, \varphi_2\}$. Arguing as in part~\textnormal{(i)}, we have $v_t \in D(I_{K^{\varphi_2}})$. Setting $w_t = \bigl( (u + t\psi) - \varphi_2 \bigr)^{+}$, we obtain $v_t - u = t\psi - w_t$. 
Using Lemma~\ref{lem:variational_inequality_characterization} together with the fact that $\varphi_2$ is a supersolution, we derive inequality~\eqref{eq:limit_variational_inequality} in the same way as before. Hence $u$ is a weak solution of~\eqref{eq:main_problem}.

\medskip

\textnormal{(iii)} We observe that $\varphi_1 \le u \le \varphi_2$, $g(\cdot,\varphi_1),\, g(\cdot,\varphi_2) \in L_{loc}^{1}(\Omega)$, and for almost every $z \in \Omega$ the function $G(z,\cdot)$ is increasing on the interval $[\varphi_1(z), \varphi_2(z)]$. 
By Lemma~\ref{lem:measurability_properties_G_g}\textnormal{(ii)}, it follows that $g(\cdot,u)$ is measurable, and $g(\cdot,u) \in L_{loc}^{1}(\Omega)$.
Fix $\psi \in C_c^{\infty}(\Omega)$ and let $t \in (0,1]$. Define $v_t = \min\{ \max \{ u + t \psi , \varphi_1\},\varphi_2\}$. Since $\varphi_1 \le v_t \le \varphi_2$, by Lemma~\ref{lem:property_G_1} and assumption~\textnormal{(A2)} we have $G(\cdot,v_t) \in L_{loc}^{1}(\Omega)$. Moreover, since $G(\cdot,u) \in L^{1}(\Omega)$ and $G(z,v_t) = G(z,u)$ on $\Omega \setminus \operatorname{supp}\psi$, we deduce that $G(\cdot,v_t) \in L^{1}(\Omega)$, that is, $v_t \in D(I_{K_{\varphi_1}^{\varphi_2}})$.
Setting
\begin{equation*}
\zeta_t = \bigl( \varphi_1 - (u + t\psi) \bigr)^{+},
\qquad
w_t = \bigl( (u + t\psi) - \varphi_2 \bigr)^{+},
\end{equation*}
we obtain $v_t - u = t\psi + \zeta_t - w_t$. Using Lemma~\ref{lem:variational_inequality_characterization} and the fact that $\varphi_1$ and $\varphi_2$ are, respectively, a subsolution and a supersolution, we can again derive inequality~\eqref{eq:limit_variational_inequality}. Consequently, $u$ is a weak solution of~\eqref{eq:main_problem}.
\end{proof}

\textbf{(A3)} Let $1 \le \bar{p} \le 2^{*}_{\gamma} - 1$. 
Assume that there exist a function $a_4$ and a constant $c_2$ such that
\begin{equation*}
\frac{f(z,u) - f(z,v)}{u - v}
\le a_4(z) + c_2 \bigl( \max \{ |u|, |v|\} \bigr)^{\bar{p}-1}
\end{equation*}
for almost every $z \in \Omega$ and for all $u, v \in \mathbb{R}$ with $u \ne v$, where $a_4 \in L^{(\bar{p}+1)/(\bar{p}-1)}(\Omega)$ is a nonnegative function and $c_2 \geq 0$.
Now, we have a Brezis-Nirenberg type theorem for the Grushin operator in the nonsmooth analysis setup inspired from the \cite[Theorem $3$]{hirano_concave_convex}.

\begin{theorem}
\label{thm:local_minimizer_from_supersolution}
Suppose that assumptions \textnormal{(A1)}-\textnormal{(A3)} are satisfied. Let $\varphi_1,\varphi_2:\Omega \to [-\infty,\infty]$ be such that $\varphi_1 \leq \varphi_2$, and assume that $\varphi_2$ is a supersolution of~\eqref{eq:main_problem}. Let $u \in D(I_{K_{\varphi_1}^{\varphi_2}})$ be a minimizer of the functional $I_{K_{\varphi_1}^{\varphi_2}}$. In addition, suppose that one of the following holds:
\begin{itemize}
    \item[(i)] $u(z) < \varphi_2(z)$ for almost every $z \in \Omega$;
    \item[(ii)] $\varphi_2$ is a strict supersolution of~\eqref{eq:main_problem}.
\end{itemize}
Then $u$ is a local minimizer of $I_{K_{\varphi_1}}$. Furthermore, if condition \textnormal{(ii)} is satisfied, and $\bar p < 2_{\gamma}^{*}-1$, then there exists $\rho_0>0$ such that
\begin{equation}
\label{eq:strict_local_minimum_estimate}
I_{K_{\varphi_1}}(u)
<
\inf \left\{
I_{K_{\varphi_1}}(v):\, v \in K_{\varphi_1},\ 
\bigl\|(v-\varphi_2)^{+}\bigr\|=\rho
\right\}
\quad \text{for every } \rho \in (0,\rho_0].
\end{equation}
\end{theorem}

\begin{proof}
We define the truncation operator
\begin{equation*}
\pi(v) = \min\{ v, \varphi_2\} = v - (v-\varphi_2)^{+},
\end{equation*}
and, for $z \in \Omega$,
\begin{equation*}
\eta_{v}(z) = \bigl( a_4(z) + c_2 ( \max\{ |v(x)|, |\varphi_2(x)|\} ^{\bar{p}-1} \bigr)\, \chi_{\{v>\varphi_2\}}(x),
\end{equation*}
for any $v \in K_{\varphi_1}$ with $\chi_{v > \phi_2}$ denotes the characteristic function. We also introduce the linear functional
\begin{equation*}
\mathcal{L}w = \int_{\Omega} \bigl( \gradgr \varphi_2 \cdot \gradgr w + g(z,\varphi_2)\, w - f(z,\varphi_2)\, w \bigr)\, dz,
\end{equation*}
for any $w \in \spacegr$ with $w \ge 0$.

As $G(\cdot, \pi(v)(\cdot)), G(\cdot, v(\cdot)) \in \real$, by Lemma~\ref{lem:measurability_properties_G_g}\textnormal{(i)} and \textnormal{(iii)}, the functions $G(\cdot,\pi(v)(\cdot))$ and $g(\cdot,\pi(v)(\cdot))\,(v(\cdot)-\pi(v)(\cdot))$ are measurable. 
Since $\pi(v) \in K_{\varphi_1}^{\varphi_2}$, using the minimality of $u$ in $K_{\varphi_1}^{\varphi_2}$, convexity of $G(z,\cdot)$ and assumption~\textnormal{(A3)}, we obtain
\begin{equation}
\label{eq:IK_key_inequality}
\begin{aligned}
I_{K_{\varphi_1}}(v) - I_{K_{\varphi_1}}(u) &\ge  I_{K_{\varphi_1}}(v) - I_{K_{\varphi_1}}(\pi(v)) \\
& = \frac{1}{2} \| v - \pi(v) \|^{2}
+ \int_{\Omega} \gradgr \pi(v) \cdot \gradgr (v-\pi(v))\, dz \\
& \quad + \int_{\Omega} \bigl( G(z,v) - G(z,\pi(v)) \bigr)\, dz
- \int_{\Omega} \bigl( F(z,v) - F(z,\pi(v)) \bigr)\, dz \\
& \ge  \frac{1}{2} \, \| v - \pi(v) \|^{2}
+ \int_{\Omega} \gradgr \pi(v) \cdot \gradgr (v-\pi(v))\, dz \\
& \quad + \int_{\Omega} g(z,\pi(v))\, (v-\pi(v))\, dz
- \int_{\Omega} f(z,\pi(v))\, (v-\pi(v))\, dz \\
& \quad  - \int_{\Omega} \Bigl( F(z,v) - F(z,\pi(v)) - f(z,\pi(v))\, (v-\pi(v)) \Bigr)\, dz \\
& = \frac{1}{2} \| v - \pi(v) \|^{2}
+ \int_{\Omega} \gradgr \varphi_2 \cdot \gradgr (v-\pi(v))\, dz \\
& \quad + \int_{\Omega} g(z,\varphi_2)\, (v-\pi(v))\, dz
- \int_{\Omega} f(z,\varphi_2)\, (v-\pi(v))\, dz \\
& \quad - \int_{\Omega} \int_{\pi(v)}^{v} \bigl( f(z,t) - f(z,\pi(v)) \bigr)\, dt \, dz \\
& \ge \frac{1}{2} \| v - \pi(v) \|^{2} + L\bigl( v - \pi(v) \bigr)
- \frac{1}{2} \int_{\Omega} \eta_{v}(z)\, (v-\pi(v))^{2}\, dz .
\end{aligned}
\end{equation}

First, we consider the theorem under assumption~\textnormal{(i)}. We show that $\| \eta_{v} \|_{(\bar{p}+1)/(\bar{p}-1)} \to 0$ as $\lv u - v \rv \to 0$ and this will prove our assertion. Suppose, on the contrary, that it fails. Then there exists a sequence $\{v_k\} \subset K_{\varphi_1}$ such that $\| v_k - u \| \le 1/k$ and $\| \eta_{v_k} \|_{(\bar{p}+1)/(\bar{p}-1)} > C_1$ for some constant $C_1 > 0$.
Using the Holder inequality and the fact that $| \{ v_k > \varphi_2\}| \to 0 $ as $k \to \infty$, we get
\begin{equation*}
\begin{aligned}
\| \eta_{v_k} \|_{\frac{\bar{p}+1}{\bar{p}-1}}& \leq  \left( \int_{\{ v_k > \varphi_2 \}} | a_4(z) |^{\frac{\bar{p}+1}{\bar{p}-1}} \, dz \right)^{\frac{\bar{p}-1}{\bar{p}+1}}  \\
& \quad + c_2 \left( \left( \int_{\{ v_k > \varphi_2 \}} | v_k(z) |^{\bar{p}+1} \, dz \right)^{\frac{1}{\bar{p}+1}} + \left( \int_{\{ v_k > \varphi_2 \}} | \varphi_2(z) |^{\bar{p}+1} \, dz \right)^{\frac{1}{\bar{p}+1}}
\right)^{\bar{p}-1}\\
& \leq  \left( \int_{\{ v_k > \varphi_2 \}} | a_4(z) |^{\frac{\bar{p}+1}{\bar{p}-1}} \, dz \right)^{\frac{\bar{p}-1}{\bar{p}+1}} \\
& {} \quad + c_2 \left(
\frac{C}{k}
+ \left( \int_{\{ v_k > \varphi_2 \}} | u(z) |^{\bar{p}+1} \, dz \right)^{\frac{1}{\bar{p}+1}}
+ \left( \int_{\{ v_k > \varphi_2 \}} | \varphi_2(z) |^{\bar{p}+1} \, dz \right)^{\frac{1}{\bar{p}+1}}
\right)^{\bar{p}-1}.
\end{aligned}
\end{equation*}
Using the Fatau's Lemma in the above, we have our assertion. Hence, $u$ is a local minimizer of $I_{K_{\varphi_1}}$.
\medskip

Next, we consider the theorem under assumption~\textnormal{(ii)} in the subcritical case, i.e., $1 \le \bar{p} < 2_{\gamma}^{*}-1$. Then there exists $M>0$ such that
\begin{equation*}
\| \eta_{v} \|_{(\bar{p}+1)/(\bar{p}-1)} \le M
\end{equation*}
for each $v \in K_{\varphi_1}$ with $\bigl\| (v-\varphi_2)^{+} \bigr\| \le 1$. Using~\eqref{eq:IK_key_inequality}, we obtain
\begin{equation*}
I_{K_{\varphi_1}}(v) - I_{K_{\varphi_1}}(u)
\ge \frac{1}{2} \bigl\| (v-\varphi_2)^{+} \bigr\|^{2}
+ \mathcal{L}\bigl( (v-\varphi_2)^{+} \bigr)
- \frac{M}{2} \bigl\| (v-\varphi_2)^{+} \bigr\|_{\bar{p}+1}^{2},
\end{equation*}
for each $v \in K_{\varphi_1}$ with $\bigl\| (v-\varphi_2)^{+} \bigr\| \le 1$. \\
If $\bigl\| (v-\varphi_2)^{+} \bigr\|^{2} \ge 2M \bigl\| (v-\varphi_2)^{+} \bigr\|_{\bar{p}+1}^{2}$, then 
\begin{equation}
\label{eq:case_one_conclusion}
I_{K_{\varphi_1}}(v) - I_{K_{\varphi_1}}(u)
\ge \frac{1}{4} \bigl\| (v-\varphi_2)^{+} \bigr\|^{2} + \mathcal{L}((v-\varphi_2)) 
\ge \frac{1}{4} \bigl\| (v-\varphi_2)^{+} \bigr\|^{2}.
\end{equation}
On the other hand, if $\bigl\| (v-\varphi_2)^{+} \bigr\|^{2} \le 2M \bigl\| (v-\varphi_2)^{+} \bigr\|_{\bar{p}+1}^{2}$, set
\begin{equation*}
\nu = \inf \{ \mathcal{L}w : w \in \mathcal{B} \},
\end{equation*}
where $\mathcal{B} = \{ w \in \spacegr : w \ge 0,\ \| w \|_{\bar{p}+1} = 1,\ \| w \| \le \sqrt{2M} \}$. Since $g(z,\varphi_2)^{-} \in L^{\frac{2Q}{Q+2}}(\Omega)$, by dominated convergence theorem and Fatau's Lemma the functional $\mathcal{L}$ is weakly lower semicontinuous on $\mathcal{B}$. But as $\mathcal{B}$ is weakly compact and  $\varphi_2$ is a strict supersolution of~\eqref{eq:main_problem}. This yields $\nu>0$. \\Choose $\rho_0 \in (0,1]$ such that, for each $v \in K_{\varphi_1}$ with $\bigl\| (v-\varphi_2)^{+} \bigr\| \le \rho_0$, one has $\bigl\| (v-\varphi_2)^{+} \bigr\|_{\bar{p}+1} \le \frac{\nu}{M}$. Fix $\rho \in (0,\rho_0]$ and take $v \in K_{\varphi_1}$ with $\bigl\| (v-\varphi_2)^{+} \bigr\| = \rho$. Then
\begin{equation}
\label{eq:case_two_conclusion}
\begin{aligned}
I_{K_{\varphi_1}}(v) - I_{K_{\varphi_1}}(u)
& \ge \left( \nu - \frac{M}{2} \bigl\| (v-\varphi_2)^{+} \bigr\|_{\bar{p}+1} \right)
      \bigl\| (v-\varphi_2)^{+} \bigr\|_{\bar{p}+1} \ge \frac{\nu}{2\sqrt{2M}}\, \rho .
\end{aligned}
\end{equation}
From~\eqref{eq:case_one_conclusion} and~\eqref{eq:case_two_conclusion}, we obtain the desired conclusion.

\medskip

Finally, under assumption~\textnormal{(ii)}, we consider the case $\bar{p} = 2_{\gamma}^{\ast} - 1$. 
Assume by contradiction that the conclusion fails. Then there exists a sequence $\{v_k\} \subset \spacegr$ such that
\begin{equation*}
v_k \ge \varphi_1, \qquad \| v_k - u \| \le 2^{-k}, \qquad I_{K_{\varphi_1}}(v_k) < I_{K_{\varphi_1}}(u)
\quad \text{for all } k .
\end{equation*}
Define
\begin{equation*}
h = u + \sum_{k=1}^{\infty} | v_k - u | ,
\end{equation*}
so that $|v_k| \le h$ almost everywhere for all $k$. Setting
\begin{equation*}
\tilde{\eta}_v(z) = \bigl( a_4(z) + c_2 ( \max \{ h(z), |\varphi_2(z)| \}^{2_{\gamma}^{\ast}-2} \bigr)\, \chi_{\{ v > \varphi_2 \}}(z),
\end{equation*}
for every $v \in D(I_{K_{\varphi_1}})$, we infer
\begin{equation*}
\begin{aligned}
0 &> I_{K_{\varphi_1}}(v_k) - I_{K_{\varphi_1}}(u)
\ge I_{K_{\varphi_1}}(v_k) - I_{K_{\varphi_1}}(\pi(v_k)) \\
&\ge \frac{1}{2} \bigl\| (v_k-\varphi_2)^{+} \bigr\|^{2}
+ \mathcal{L}\bigl( (v_k-\varphi_2)^{+} \bigr)
- \frac{1}{2} \int_{\Omega} \tilde{\eta}_{v_k}(z)\, \bigl( (v_k-\varphi_2)^{+} \bigr)^{2}\, dz \\
& \ge \frac{1}{2} \bigl\| (v_k-\varphi_2)^{+} \bigr\|^{2}
+ \mathcal{L}\bigl( (v_k-\varphi_2)^{+} \bigr)
- \frac{1}{2} \int_{\{ \tilde{\eta}_{v_k} \le \tilde{M} \}} \tilde{\eta}_{v_k}(z)\, \bigl( (v_k-\varphi_2)^{+} \bigr)^{2}\, dz \\
& \qquad
- \frac{1}{2} \int_{\{ \tilde{\eta}_{v_k} > \tilde{M} \}} \tilde{\eta}_{v_k}(z)\, \bigl( (v_k-\varphi_2)^{+} \bigr)^{2}\, dz \\
& \ge \frac{1}{2} \bigl\| (v_k-\varphi_2)^{+} \bigr\|^{2}
+ \mathcal{L}\bigl( (v_k-\varphi_2)^{+} \bigr)
- \frac{\tilde{M}}{2} \int_{\Omega} \bigl| (v_k-\varphi_2)^{+} \bigr|^{2}\, dz \\
& \qquad
- \frac{1}{2S} \left( \int_{\{ \tilde{\eta}_{v_k} > \tilde{M} \}} \lvert \tilde{\eta}_{v_k}(z) \rvert^{\frac{2_{\gamma}^{\ast}}{2_{\gamma}^{\ast}-2}}\, dz \right)^{\frac{2_{\gamma}^{\ast}-2}{2_{\gamma}^{\ast}}}
\bigl\| (v_k-\varphi_2)^{+} \bigr\|^{2},
\end{aligned}
\end{equation*}
where $S$ denotes the Sobolev constant. Since we can choose $\tilde{M} > 0$ such that
\begin{equation*}
\frac{1}{2S}
\left(
\int_{\{ \tilde{\eta}_{v_k} > \tilde{M} \}}
\lvert \tilde{\eta}_{v_k}(z) \rvert^{\frac{2_{\gamma}^{\ast}}{2_{\gamma}^{\ast}-2}}\, dz
\right)^{\frac{2_{\gamma}^{\ast}-2}{2_{\gamma}^{\ast}}}
< \frac{1}{4}
\end{equation*}
for all $k$, we obtain
\begin{equation*}
0 >
\mathcal{L}\bigl( (v_k-\varphi_2)^{+} \bigr)
+ \frac{1}{4} \bigl\| (v_k-\varphi_2)^{+} \bigr\|^{2}
- \frac{\tilde{M}}{2} \bigl\| (v_k-\varphi_2)^{+} \bigr\|_{2}^{2}
\quad \text{for all } k .
\end{equation*}
If $\bigl\| (v_k-\varphi_2)^{+} \bigr\|^{2}
\ge 4\tilde{M}\, \bigl\| (v_k-\varphi_2)^{+} \bigr\|_{2}^{2}$, then 
\begin{equation*}
0 >
\frac{1}{4} \bigl\| (v_k-\varphi_2)^{+} \bigr\|^{2}
- \frac{1}{8} \bigl\| (v_k-\varphi_2)^{+} \bigr\|^{2}
= \frac{1}{8} \bigl\| (v_k-\varphi_2)^{+} \bigr\|^{2}.
\end{equation*}
On the other hand, if $\bigl\| (v_k-\varphi_2)^{+} \bigr\|^{2}
\le 4\tilde{M}\, \bigl\| (v_k-\varphi_2)^{+} \bigr\|_{2}^{2}$, we define
\begin{equation*}
\tilde{\nu} = \inf \{ \mathcal{L}w : w \in \tilde{\mathcal{B}} \},
\end{equation*}
where $\tilde{\mathcal{B}} = \{ w \in \spacegr : w \ge 0,\ \| w \|_{2} = 1,\ \| w \| \le 2\sqrt{\tilde{M}} \}$. As in the previous case, we have $\tilde{\nu} > 0$. Fix $k$ such that
\begin{equation*}
\bigl\| (v_k-\varphi_2)^{+} \bigr\|_{2} \le \frac{\tilde{\nu}}{\tilde{M}} .
\end{equation*}
Then
\begin{equation*}
0 >
\left(
\tilde{\nu} - \frac{\tilde{M}}{2} \bigl\| (v_k-\varphi_2)^{+} \bigr\|_{2}
\right)
\bigl\| (v_k-\varphi_2)^{+} \bigr\|_{2}
\ge
\frac{\tilde{\nu}}{4\sqrt{\tilde{M}}}\, \bigl\| (v_k-\varphi_2)^{+} \bigr\|.
\end{equation*}
In both cases we reach a contradiction. 
\end{proof}

As a dual statement to the previous theorem, we obtain the following result.

\begin{theorem}
\label{thm:dual_local_minimizer}
Suppose that assumptions \textnormal{(A1)}--\textnormal{(A3)} hold. Let $\varphi_1,\varphi_2:\Omega \to [-\infty,\infty]$ be such that $\varphi_1 \leq \varphi_2$, and assume that $\varphi_1$ is a subsolution of~\eqref{eq:main_problem}. Let $u \in D(I_{K_{\varphi_1}^{\varphi_2}})$ be a minimizer of the functional $I_{K_{\varphi_1}^{\varphi_2}}$. In addition, suppose that one of the following conditions is satisfied:
\begin{itemize}
    \item[(i)] $u(z) > \varphi_1(z)$ for almost every $z \in \Omega$;
    \item[(ii)] $\varphi_1$ is a strict subsolution of~\eqref{eq:main_problem}.
\end{itemize}
Then $u$ is a local minimizer of $I_{K^{\varphi_2}}$. Furthermore, if condition \textnormal{(ii)} holds and $\bar p < 2_{\gamma}^{\ast}-1$, then there exists $\rho_0>0$ such that
\begin{equation*}
I_{K^{\varphi_2}}(u)
<
\inf \left\{
I_{K^{\varphi_2}}(v):\, v \in K^{\varphi_2},\ 
\bigl\|(\varphi_1-v)^{+}\bigr\|=\rho
\right\}
\quad \text{for all } \rho \in (0,\rho_0].
\end{equation*}
\end{theorem}

Now, with the help of the Theorem \ref{thm:local_minimizer_from_supersolution} and Theorem \ref{thm:dual_local_minimizer}, we will prove a Brezis-Nirenberg type result using the Perron-type method inspired from Theorem $4$ of \cite{hirano_concave_convex}.

\begin{theorem}
\label{thm:existence_between_sub_super}
Assume that \textnormal{(A1)}-\textnormal{(A3)} are satisfied. Let $\varphi_1$ and $\varphi_2$ be, respectively, a subsolution and a supersolution of~\eqref{eq:main_problem}, and suppose that $\varphi_2$ is not a solution of~\eqref{eq:main_problem}. Assume further that $\varphi_1,\varphi_2 \in L^{p+1}(\Omega)$, that $G(\cdot,\varphi_1),\,G(\cdot,\varphi_2)\in L_{loc}^{1}(\Omega)$, and that $\varphi_1 \leq \varphi_2$ almost everywhere in $\Omega$. 

Suppose there exists $v \in \spacegr$ such that $\varphi_1 \leq v \leq \varphi_2$ and $G(\cdot,v)\in L^{1}(\Omega)$. In addition, assume that one of the following conditions holds:
\begin{itemize}
    \item[(i)] For every subdomain $\Omega' \Subset \Omega$, there exists a constant $M>0$ such that
    \begin{equation*}
    \frac{g\bigl(z,\varphi_2(z)\bigr)-g(z,t)}{\varphi_2(z)-t}\leq M
    \quad \text{and} \quad
    \frac{f\bigl(z,\varphi_2(z)\bigr)-f(z,t)}{\varphi_2(z)-t}\geq -M
    \end{equation*}
    for all $(z,t)\in \Omega'\times\mathbb{R}$ satisfying $\varphi_1(z)\leq t<\varphi_2(z)$;
    
    \item[(ii)] $\varphi_2$ is a strict supersolution of~\eqref{eq:main_problem}.
\end{itemize}

Then there exists a weak solution $u$ of~\eqref{eq:main_problem} such that
\[
\varphi_1 \leq u \leq \varphi_2
\quad \text{almost everywhere in } \Omega,
\]
and $u$ is a local minimizer of $I_{K_{\varphi_1}}$. Moreover, if condition \textnormal{(i)} holds, then $u < \varphi_2$ in $\om$. If condition \textnormal{(ii)} is satisfied and $\bar p < 2_{\gamma}^{\ast}-1$, then there exists $\rho_0>0$ such that condition~\eqref{eq:strict_local_minimum_estimate} is valid.
\end{theorem}
\begin{proof}
Since $G(\cdot, v) \in L^1(\om)$, we have $\inf I_{K_{\varphi_1}^{\varphi_2}}\!\left( K_{\varphi_1}^{\varphi_2} \right) < \infty.$
Let $\{u_k\} \subset K_{\varphi_1}^{\varphi_2}$ be the minimizing sequence, i.e., $I_{K_{\varphi_1}^{\varphi_2}}(u_k) \downarrow \inf I_{K_{\varphi_1}^{\varphi_2}}\!\left( K_{\varphi_1}^{\varphi_2} \right)$.
By \textnormal{(A1)} and \textnormal{(A2)}, the sequence $\{ \|u_k\| \}$ is bounded. Hence, up to a subsequence, we may assume that $u_k \rightharpoonup u$ in $\spacegr$ and $u_k \to u$ almost everywhere in $\Omega$. Moreover, as $K_{\varphi_1}^{\varphi_2}$ is weakly closed we have $u \in K_{\varphi_1}^{\varphi_2}$.

Since $\{u_k\} \subset K_{\varphi_1}^{\varphi_2}$, $\varphi_1, \varphi_2 \in L^{p+1}(\Omega)$, and~\eqref{eq:main_problem} holds, Lebesgue's dominated convergence theorem yields
\begin{equation*}
\int_{\Omega} f(z,u)\, dz = \lim_{n \to \infty} \int_{\Omega} F(x,u_k)\, dz .
\end{equation*}
Moreover, the mappings $v \mapsto \int_{\Omega} |\nabla v|^{2} \, dz$ and $v \mapsto \int_{\Omega} G(z,v)\, dz$ are weakly sequentially lower semicontinuous on $\spacegr$. Therefore,
\begin{equation*}
I_{K_{\varphi_1}^{\varphi_2}}(u) \le \liminf_{n \to \infty} I_{K_{\varphi_1}^{\varphi_2}}(u_k),
\end{equation*}
and hence $u \in K_{\varphi_1}^{\varphi_2}$ is a minimizer of $I_{K_{\varphi_1}^{\varphi_2}}$. In particular, $0 \in \partial^{-} I_{K_{\varphi_1}^{\varphi_2}}(u)$. By Proposition~2\,(iii), $u$ is a weak solution of~\eqref{eq:main_problem}.

Define the function $h : \Omega \to \mathbb{R}$ by
\begin{equation*}
h(z) =
\begin{cases}
\dfrac{\bigl( g(z,\varphi_2(z)) - g(z,u(z)) \bigr)
      + \bigl( f(z,\varphi_2(z)) - f(z,u(z)) \bigr)}
      {\varphi_2(z) - u(z)},
& \text{if } \varphi_2(z) > u(z), \\[1.2ex]
0, & \text{if } \varphi_2(z) = u(z),
\end{cases}
\end{equation*}
for every $z \in \Omega$. By assumption~\textnormal{(i)}, we have $h \in L_{loc}^{\infty}(\Omega)$. Since
\begin{equation*}
- \Gr(\varphi_2 - u) + h\,(\varphi_2 - u) \ge 0
\quad \text{in } \Omega
\end{equation*}
in the sense of distributions, the strong maximum principle in Lemma \ref{lem:SMP_grushin} implies $u < \varphi_2$ in $\om$.  Thus, under assumption~\textnormal{(i)}, we have verified condition~\textnormal{(i)} of Theorem \ref{thm:local_minimizer_from_supersolution}. Consequently, by Theorem \ref{thm:local_minimizer_from_supersolution}, the conclusion follows. \\
\end{proof}
\medskip
\subsubsection{Supercritical Problem}

We end this section by proving an existence result when the nonlinearity behaves like a supercritical term, i.e., $N  > \frac{Q+2}{Q-2}$. For this, we need a different set of assumptions given by

\textbf{(A1')} Let $f$ be a Carath\'{e}odary function such that for every $M>0$ there exists a nonnegative function $a_4 \in L^{2Q/(Q+2)}(\om)$ such that
\begin{equation*}
|f(z,u)| \leq a_4(z)
\end{equation*}
for a.e. $z \in \om$ and $u(z) \in \real $ with $|u(z)|\leq M$.

\textbf{(A3')} For every $M>0$ there exists a nonnegative function $a_5 \in L^{Q/2}(\om)$ such that 
\begin{equation*}
    \frac{f(z,u) - f(z,v)}{u-v} \leq a_5(z)
\end{equation*}
for a.e. $z \in \om$ and for every $u,v \in \real$ with $u\neq v$ and $|u|,|v| \leq M$. \\

Let $f(\cdot,t) = \frac{\partial}{\partial t} F(\cdot, t)$ as before. Since the function $F(\cdot, u)$ may not belong to $L^1(\om)$ for every $u \in \spacegr$, we need to redefine the functional associated with the problem \eqref{eq:main_problem}. Consider $I : \spacegr \to (-\infty , \infty ]$ defined as 
\begin{equation*}
I(u) =
\begin{cases}
\dfrac{1}{2} \displaystyle\int_{\Omega} |\gradgr u|^{2}\,dz
+ \displaystyle\int_{\Omega} G(z,u)\,dz
- \displaystyle\int_{\Omega} F(z,u)\,dz,
& \text{if } G(\cdot,u), F(\cdot,u) \in L^{1}(\Omega), \\[1.2ex]
\infty, & \text{otherwise}.
\end{cases}
\end{equation*}
Then we have the analogous version of Theorem \ref{thm:existence_between_sub_super} for the supercritical case is given by
\begin{theorem}
\label{thm:local_minimizer_from_supersolution_supercritical}
Assume (A1'), (A2), and (A3'). Let $\phi_1$ and $\phi_2$ be two functions satisfying the following assumptions
\begin{enumerate}
    \item[$(i)$] $\phi_1$ is a subsolution and $\phi_2$ is a supersolution of \eqref{eq:main_problem} respectively,
    \item[$(ii)$] $\phi_1, \phi_2 \in L^{\infty}(\om)$ with $\phi_1 \leq \phi_2$,
    \item[$(iii)$] $G(\cdot, \phi_1), G(\cdot, \phi_2) \in L_{loc}^1(\om)$,
    \item[$(iv)$] there exists $w \in \spacegr$ such that $\phi_1 \leq w \leq \phi_2$ and $G(\cdot, w) \in L^1(\om)$,
    \item[$(v)$] either $(i)$ or $(ii)$ of Theorem \ref{thm:existence_between_sub_super} hold.
\end{enumerate}
Then there exists a weak solution $u$ of \eqref{eq:main_problem} satisfying the following 
\begin{itemize}
    \item $\phi_1 \leq u \leq \phi_2$,
    \item for each $v \in \spacegr \cap L^{\infty}(\om)$ with $v \geq u$, $u$ minimizes the functional $I_{K_{\phi_1}^v}$,
    \item if $(i)$ of Theorem \ref{thm:existence_between_sub_super} holds, then $u < \phi_2$ in $\om$.
\end{itemize}
\end{theorem}
\begin{proof}
Since the functional $I$ is lower semi-continuous on $K_{\phi_1}^{\phi_2}$ with the virtue of the assumption that $\phi_1, \phi_2 \in L^{\infty}(\om)$, there exists a minimizer $u \in K_{\phi_1}^{\phi_2}$ of $I$ in the convex set $K_{\phi_1}^{\phi_2}$. \\
Now, let $v \in K_u \cap L^{\infty}(\om)$ and define $\tilde{f}(z,s) = f(z,\max \{ \phi_1, \min\{ s, \max \{ v, \phi_2\} \} \})$ for all $(z,s) \in \om \times \real$ with $\tilde{I}_K$ defined as in \eqref{eq:definition_of_IK} where $f$ is replaced with $\tilde{f}$. We show that all the assumptions of Theorem \ref{thm:local_minimizer_from_supersolution} are satisfied. Since 
\begin{equation*}
\max \{ \phi_1, \min\{ s, \max \{ v, \phi_2\} \} \} = s \quad 
\end{equation*}
for every $(z,s) \in \om \times \real$ with $\phi_1(z) \leq z \leq \max \{\phi_2(z), v(z)\}$.
It is clear that $u$ is a minimizer of $\tilde{I}_{K_{\phi_1}^{\phi_2}}$. Also, $\phi_2$ is a supersolution for \eqref{eq:main_problem} with $\tilde{f}$ replaced with $f$. Further, both the assumptions $(i)$ and $(ii)$ of Theorem \ref{thm:existence_between_sub_super} holds with $\tilde{f}$ replaced with $f$. Since all the assumptions of Theorem \ref{thm:local_minimizer_from_supersolution} are satisfied, $u$ is a local minimizer for $\tilde{I}_{K_{\phi_1}}$. This yields that $u$ is a local minimizer for $I_{K_{\phi_1}^v}$.
\end{proof}

\section{Purely Singular Problem}\label{sec:purely_singular}

In this section, we study the purely singular version of the problem \eqref{eq:main_problem}. In particular, we are interested in the following singular boundary value problem:
\begin{equation}\tag{$\mathcal{S}_0$}
\label{eq:singular_grushin}
\begin{cases}
-\Gr u_0 = u_0^{-\delta}, & \text{in } \Omega, \\
u_0 > 0, & \text{in } \Omega, \\
u_0 = 0, & \text{on } \partial \Omega.
\end{cases}
\end{equation}
We establish the existence of solutions to \eqref{eq:singular_grushin} together with their asymptotic estimates, as stated in the following result.

\begin{theorem}[Existence and uniqueness]
\label{thm:singular-problem}
There exists an unique function $u_0 \in C_{loc}^{0,\alpha}(\om)$ satisfying \eqref{eq:singular_grushin}. Moreover, if $u_1 \in \spacegr \cap C_{loc}^{0,\alpha}(\om) \cap L^{\infty}(\om)$ satisfies
\begin{equation}\tag{$\mathcal{S}'$}
\label{eq:one_grushin}
\begin{cases}
-\Gr u_1 = 1, & \text{in } \Omega, \\
u_1 = 0, & \text{on } \partial \Omega,
\end{cases}
\end{equation}
then 
\begin{equation*}
\|u_1\|_\infty^{-\frac{\delta}{1+\delta}}\, u_1
\le
u_0
\le
\bigl( (1+\delta) u_1 \bigr)^{\frac{1}{1+\delta}}
\qquad \text{in } \Omega .
\end{equation*}
\end{theorem}
\begin{remark}
\begin{enumerate}
    \item The existence of $u_1 \in \spacegr$ in the above theorem follows from the Lax-Milgram theorem. Also, $u_1 \in C_{loc}^{0,\alpha}(\om)$ from \cite[Theorem $5.5$]{lanconelli_holder_Xelliptic} and $u_1 \in L^{\infty}(\om)$ from Lemma \ref{lem:uniform_estimate_grushin}.
    \item If $u_1 \in C_{loc}^{0,\alpha}(\om)$, then from the inequality $(a+b)^q - a^q \leq b^q$ for $0 < q < 1$ and $a,b >0$, we have $u_1^{\frac{1}{1+\delta}} \in C_{loc}^{0,\frac{\alpha}{1 + \delta}}(\om)$.
\end{enumerate}
\end{remark}
To establish the above theorem, we first require two auxiliary lemmas. These results are primarily based on Perron's method, namely, the existence of a solution lying between a subsolution and a supersolution, together with a comparison principle. Since both lemmas can be proved in a more general framework, we consider the equation
\begin{equation}
\label{eq:main_eq_singular}
- \Delta_{\gamma} u = g(z,u) + w \qquad \text{in } \Omega,
\end{equation}
where $g : \Omega \times \mathbb{R} \to \mathbb{R}$ is a Carath\'{e}odory function and $w \in H^{-1,2}_{\gamma}(\Omega)$.

The notions of subsolution and supersolution for \eqref{eq:main_eq_singular} are defined analogously to those in \eqref{def:weak_sub_super_solution_grushin}. It is worth noting that the class of test functions appearing in the definition of weak subsolutions and supersolutions may be enlarged from $\phi \in C_c^{\infty}(\Omega)$ to $\phi \in \spacegr \cap L^\infty_{loc}(\Omega)$.

Since our interest lies in the collection of all functions bounded between a subsolution and a supersolution, which forms a convex set, we are naturally led to the following inequality:

\begin{definition}
Let $w \in H^{-1,2}_{\gamma}(\Omega)$ and $\varphi \in H^{1}_{\gamma, loc}(\Omega)$. We consider the variational inequality
\begin{equation}
\label{eq:variational-inequality}
\int_\Omega
\gradgr u \cdot \gradgr (v-u)\,dz
\ge
\int_\Omega
u^{-\delta}(v-u)\,dz
+
\langle w, v-u \rangle,
\qquad
\text{ for all }\, v \ge 0 .
\end{equation}
\textbf{Local subsolution(supersolution).}
We say that $\varphi$ is a \emph{local subsolution} of \eqref{eq:variational-inequality} if the following conditions are satisfied:
\begin{enumerate}
\item $\varphi > 0$ a.e.\ in $\Omega$ and $\varphi^{-\delta} \in L^1_{loc}(\Omega)$,
\item for every $v \in \spacegr \cap L^\infty_{loc}(\Omega)$ and $0 \le v \le \varphi$ a.e. in $\om$, the following inequality holds
\begin{equation*}
\int_\Omega
\gradgr \varphi \cdot \gradgr v \, dz
\le
\int_\Omega
\varphi^{-\delta} v \, dz
+
\langle w , v \rangle.
\end{equation*}
\end{enumerate}
For supersolutions, the above inequality is reversed.
\end{definition}

Now, we are ready to state and prove the auxilliary lemmas. 
\begin{lemma}
\label{lem:variational-identity}
Let $g : \Omega \times \mathbb{R} \to \mathbb{R}$ be a Carath\'{e}odory function such that for all $S > 0$,
\begin{equation*}
\sup_{|s|\le S} |g(\cdot,s)| \in L^1_{loc}(\Omega).
\end{equation*}
Let $w \in H^{-1,2}_{\gamma}(\Omega)$, and $\varphi,\, u,\, \psi \in  H^{1}_{\gamma, loc}(\Omega)$. Assume that
\begin{itemize}
\item $\varphi$ is a subsolution of equation \eqref{eq:main_eq_singular};
\item $\psi$ is a supersolution of equation \eqref{eq:main_eq_singular};
\item$\varphi \le u \le \psi$ a.e. in $\om$;
\item$g(z,u) \in L^1_{loc}(\Omega)$;
\item for every $v \in u + \bigl( \spacegr \cap L^\infty_{loc}(\Omega) \bigr)$ with $\varphi \le v \le \psi$ a.e. in $\om$, the following holds
\begin{equation}\label{ineq:min_on_convex}
\int_\Omega
\gradgr u \cdot \gradgr (v-u)\,dz
\ge
\int_\Omega
g(z,u)(v-u)\,dz
+
\langle w , v-u \rangle.
\end{equation}
\end{itemize}
Then $u$ satisfies
\begin{equation*}
- \Delta_{\gamma} u = g(z,u) + w
\end{equation*}
in the sense of distributions. 
\end{lemma}
\begin{proof}
Let $\eta \in C_c^\infty(\mathbb{R})$ be such that
\begin{equation*}
0 \le \eta \le 1 \quad \text{in } \mathbb{R}, \qquad
\eta = 1 \ \text{on } [-1,1], \qquad
Supp(\eta) \subset (-2,2).
\end{equation*}
Let $v \in C_c^\infty(\Omega)$ with $v \ge 0$. Fix $k \ge 1$ and $t > 0$, and set
\begin{equation*}
v_k = \eta\!\left( \frac{u}{k} \right) v,
\qquad
v_{k,t} = \min \{ u + t v_k , \psi \}.
\end{equation*}
Since $u \le v_{k,t} \le \psi$, using $v_{k,t}$ as the test function in \eqref{ineq:min_on_convex}, we obtain
\begin{align*}
&\int_\Omega
\Big(
|\gradgr (v_{k,t}-u)|^2
-
\big( g(z,v_{k,t}) - g(z,u) \big)
(v_{k,t}-u)
\Big)\,dz
 \\ &
\le
\int_\Omega
\gradgr v_{k,t} \cdot \gradgr (v_{k,t}-u)
-
g(z,v_{k,t})(v_{k,t}-u)\,dz
-
\langle w , v_{k,t}-u \rangle
\\ & =
\int_\Omega
\gradgr v_{k,t} \cdot \gradgr (v_{k,t}-u-t v_k)
-
g(z,v_{k,t})(v_{k,t}-u-t v_k)\,dz
-
\langle w , v_{k,t}-u-t v_k \rangle
\\
&\quad
+ t \int_\Omega
\big(
\gradgr v_{k,t} \cdot \gradgr v_k
-
g(z,v_{k,t}) v_k
\big)\,dz
-
t \langle w , v_k \rangle \\
& =
\int_\Omega
\gradgr \psi \cdot \gradgr (v_{k,t}-u-t v_k)
-
g(z,\psi)(v_{k,t}-u-t v_k)\,dz
-
\langle w , v_{k,t}-u-t v_k \rangle
\\
&+ t \int_\Omega
\big(
\gradgr v_{k,t} \cdot \gradgr v_k
-
g(z,v_{k,t}) v_k
\big)\,dz
-
t \langle w , v_k \rangle 
\\& \leq 
t \int_\Omega
\big(
\gradgr v_{k,t} \cdot \gradgr v_k
-
g(z,v_{k,t}) v_k
\big)\,dz
-
t \langle w , v_k \rangle.
\end{align*}
where the last inequality follows from the fact that $\psi$ is a supersolution of equation \eqref{eq:main_eq_singular}.

Thus, the relation $v_{k,t} - u \le t v_k \le t v$ yields
\begin{equation*}
\int_\Omega
\big(
\gradgr v_{k,t} \cdot \gradgr v_k
-
g(z,v_{k,t}) v_k
\big)\,dz
-
\langle w , v_k \rangle
\ge
-
\int_\Omega
\big| g(z,v_{k,t}) - g(z,u) \big|\, |v_k|\, dz .
\end{equation*}

Since
\begin{equation*}
|g(z,v_{k,t})|\,|v_k|
\le
\left(
\sup_{|s|\le 2k + \|v\|_\infty}
|g(z,s)|
\right)
|v|,
\end{equation*}
we may pass to the limit as $t \to 0^{+}$ by the Dominated Convergence Theorem. Consequently, we obtain
\begin{equation*}
\int_\Omega
\big(
\gradgr u \cdot \gradgr v_k
-
g(z,u) v_k
\big)\,dz
-
\langle w , v_k \rangle
\ge 0 .
\end{equation*}
Letting $k \to \infty$, it follows that
\begin{equation}
\label{eq:ineq-ten}
\int_\Omega
\big(
\gradgr u \cdot \gradgr v
-
g(z,u) v
\big)\,dz
-
\langle w , v \rangle
\ge 0 ,
\end{equation}
for every $v \in C_c^\infty(\Omega)$ with $v \ge 0$. Now let $v \in C_c^\infty(\Omega)$ with $v \le 0$. For $k \ge 1$ and $t > 0$, define
\begin{equation*}
v_k = \eta\!\left( \frac{u}{k} \right) v,
\qquad
v_{k,t} = \max \{ u + t v_k , \varphi \}.
\end{equation*}
Repeating the previous argument, we deduce that inequality
\eqref{eq:ineq-ten} holds for every $v \in C_c^\infty(\Omega)$.

Finally, replacing $v$ by $-v$ in \eqref{eq:ineq-ten}, we conclude the proof.
\end{proof}

\begin{lemma}
\label{lem:comparison}
Let $w \in H^{-1}_{\gamma}(\Omega)$ and $\varphi,\, \psi \in H^{1}_{\gamma, loc}(\Omega)$. Assume that $\varphi$ is a subsolution of \eqref{eq:variational-inequality} such that $\varphi \le 0$ on $\partial\Omega$ and that $\psi$ is a supersolution of \eqref{eq:variational-inequality}. Then $\varphi \le \psi$ a.e. in $\om$.
\end{lemma}
\begin{proof}
Let $\varepsilon > 0$ and choose $k > \varepsilon^{-\delta}$. Let $\widehat{\Phi}_k : \mathbb{R} \to \mathbb{R}$ be a primitive of the function
\begin{equation*}
\phi_k(s) =
\begin{cases}
\max\{ -s^{-\delta}, -k \}, & \text{if } s > 0, \\[1ex]
-k, & \text{if } s \le 0 
\end{cases}
\end{equation*}
with $\widehat{\Phi}_k(1) = 0$. Define the functional $\widehat{F}_{0,k} : L^2(\Omega) \to (-\infty,+\infty]$ by
\begin{equation*}
\widehat{F}_{0,k}(u)
=
\begin{cases}
\dfrac12
\displaystyle
\int_\Omega |\gradgr u|^2\,dz
+
\int_\Omega \widehat{\Phi}_k(u)\,dz,
& \text{if } u \in \spacegr, \\[3ex]
+\infty,
& \text{if } u \in L^2(\Omega)\setminus \spacegr.
\end{cases}
\end{equation*}
We normalise the definition of $\widehat{F}_{0,k}$ by defining $F_{0,k} : L^2(\Omega) \to (-\infty,+\infty]$ as
\begin{equation*}
F_{0,k}(u)
=
\widehat{F}_{0,k}(u)
-
\min_{L^2(\Omega)} \widehat{F}_{0,k}
=
\widehat{F}_{0,k}(u)
-
\widehat{F}_{0,k}(u_{0,k}),
\end{equation*}
where $u_{0,k} \in \spacegr$ denotes a minimizer of $\widehat{F}_{0,k}$.

We define the functional $F_{w,k} : L^2(\Omega) \to (-\infty,+\infty]$ by
\begin{equation*}
F_{w,k}(u)
=
\begin{cases}
F_{0,k}(u) - \langle w , u - u_{0,k} \rangle,
& \text{if } u \in \spacegr, \\[1ex]
+\infty,
& \text{if } u \in L^2(\Omega)\setminus \spacegr.
\end{cases}
\end{equation*}
Now, let $u_k$ be a minimizer of the functional $F_{w,k}$ on the convex set $K_{0}^{\psi} = \bigl\{ u \in \spacegr : 0 \le u \le \psi \ \text{a.e.\ in } \Omega \bigr\}.$

Then by the definition of minimum, we necessarily have
\begin{equation*}
\langle F_{w,k}'(u_k), v - u_k \rangle \ge 0
\qquad
\text{ for all }\, v \in K .
\end{equation*}

That is
\begin{equation*}
\int_\Omega
\gradgr u_k \cdot \gradgr (v-u_k)\,dz
\ge
-
\int_\Omega
\widehat{\Phi}_k'(u_k)(v-u_k)\,dz
+
\langle w , v-u_k \rangle,
\qquad
\text{ for all }\, v \in K .
\end{equation*}

Let $v \in C_c^\infty(\Omega)$ with $v \ge 0$ and let $t > 0$. We consider the test function $v_t = \min\{ u_k + t v , \psi \}$. Since $\psi$ is a supersolution of $-\Delta_{\gamma} u = -\widehat{\Phi}_k'(u) + w$. Arguing as in the previous lemma, we deduce
\begin{equation}
\label{eq:ineq-eleven}
\int_\Omega
\gradgr u_k \cdot \gradgr v\,dz
\ge
-
\int_\Omega
\widehat{\Phi}_k'(u_k) v\,dz
+
\langle w , v \rangle .
\end{equation}
Choosing $v = (\varphi - u - \varepsilon)^+ \in \spacegr$ in \eqref{eq:ineq-eleven}, we obtain
\begin{equation*}\label{eq:ineq_eleven1}
\int_\Omega
\gradgr u \cdot \gradgr (\varphi - u - \varepsilon)^+ \, dz
\ge
-
\int_\Omega
\Phi_k'(u) (\varphi - u - \varepsilon)^+ \, dz
+
\langle w , (\varphi - u - \varepsilon)^+ \rangle .
\end{equation*}
Now, let $v \in \spacegr$ such that $0 \le v \le \varphi$ a.e. in $\om$ and assume moreover that $\gradgr \varphi \in L^2(\{ v > 0 \})$. Let $(\hat v_k)$ be a sequence in $C_c^\infty(\Omega)$ converging to $v$ in $\spacegr$, and define $v_k = \min\{ \hat v_k^{+}, v \}$. Then we have
\begin{equation*}
\int_\Omega
\gradgr \varphi \cdot \gradgr v_k \, dz
\le
\int_\Omega
\varphi^{-\delta} v_k \, dz
+
\langle w , v_k \rangle .
\end{equation*}
If $\varphi^{-\delta} v \notin L^1(\Omega)$, then \eqref{eq:main-ineq-13} holds trivially. If $\varphi^{-\delta} v \in L^1(\Omega)$, we may pass to the limit as
$k \to \infty$ and obtain
\begin{equation}
\label{eq:main-ineq-13}
\int_\Omega
\gradgr \varphi \cdot \gradgr v \, dz
\le
\int_\Omega
\varphi^{-\delta} v \, dz
+
\langle w , v \rangle .
\end{equation}
In particular, choosing $v = (\varphi - u - \varepsilon)^+$, we deduce
\begin{equation}
\label{eq:main-ineq-14}
\int_\Omega
\gradgr \varphi \cdot
\gradgr (\varphi - u - \varepsilon)^+ \, dz
\le
\int_\Omega
\varphi^{-\delta} (\varphi - u - \varepsilon)^+ \, dz
+
\langle w , (\varphi - u - \varepsilon)^+ \rangle .
\end{equation}
Since $\varepsilon^{-\delta} < k$, from \eqref{eq:ineq_eleven1} and \eqref{eq:main-ineq-14} we deduce that
\begin{equation*}
\begin{aligned}
\int_\Omega
\bigl| \gradgr (\varphi - u - \varepsilon)^+ \bigr|^2 \, dz
&=
\int_\Omega
\gradgr (\varphi - u)
\cdot
\gradgr (\varphi - u - \varepsilon)^+ \, dz
\\
&\le
\int_\Omega
\left( \varphi^{-\delta} + \Phi_k'(u) \right)
(\varphi - u - \varepsilon)^+ \, dz
\\
&=
\int_\Omega
\left(
- \Phi_k'(\varphi) + \Phi_k'(u)
\right)
(\varphi - u - \varepsilon)^+ \, dz \leq 0.
\end{aligned}
\end{equation*}
Hence, $(\varphi - u - \varepsilon)^+ = 0$ a.e. in $\om$. Therefore, $\varphi \le u + \varepsilon \le \psi + \varepsilon $. Letting $\varepsilon \downarrow 0$, we conclude that $\varphi \le u \le \psi$ a.e. in $\om$.
\end{proof}

\begin{proof}[Proof of Theorem \ref{thm:singular-problem}]
Let $\Phi_k : \mathbb{R} \to \mathbb{R}$ and $F_{0,k} : L^2(\Omega) \to (-\infty,+\infty]$ be defined as in the previous section. Let $u_1 \in \spacegr \cap C^{0,\alpha}_{loc}(\Omega) \cap L^{\infty}(\om)$ be the solution of
\begin{equation*}
-\Delta_{\gamma} u_1 = 1
\qquad \text{in } \Omega .
\end{equation*}
Let us define
\begin{equation*}
\varphi
=
\|u_1\|_\infty^{-\frac{\delta}{1+\delta}}\, u_1,
\qquad
\psi
=
\bigl( (1+\delta) u_1 \bigr)^{\frac{1}{1+\delta}} .
\end{equation*}
Since $u_1 > 0$ in $\Omega$ by the Strong maximum principle in Lemma \ref{lem:SMP_grushin}, it follows that $\varphi > 0$ in $\om$. Also, by definition, we have $\varphi \le \psi$. Moreover, for every $k \ge \|u_1\|_\infty^{-\frac{\delta}{1+\delta}}$, $\varphi$ is a subsolution and $\psi$ is a supersolution of 
\begin{equation}
\label{eq:regularized-problem}
\begin{cases}
-\Delta_{\gamma} u = -\Phi_k'(u) & \text{in } \Omega, \\
u = 0 & \text{on } \partial\Omega .
\end{cases}
\end{equation}
Since the functional $F_{0,k}$ is strictly convex, it admits one and only one minimum, say $u_{0,k}$ on the convex set
\begin{equation*}
K_{\phi}^{\psi} =
\bigl\{
u \in \spacegr
:
\varphi \le u \le \psi \ \text{a.e.\ in } \Omega
\bigr\}.
\end{equation*}
Let $g(z,s) = -\Phi_k'(s)$, then $g(z,u_{0,k}) \in L^1_{loc}(\Omega)$, and the following inequality holds 
\begin{equation*}
 \int_{\om} \gradgr u_{0,k} \cdot \gradgr (v - u_{0,k}) dz = - \int_{\om} \Phi_k^{'} (u_{0,k}) (v - u_{0,k}) dz,   
\end{equation*}
for all $v \in K_{\phi}^{\psi}$. Hence, by Lemma \ref{lem:variational-identity}, the function $u$ is a weak solution of \eqref{eq:regularized-problem}. Since $u_{0,k}$ is a subsolution of
\begin{equation*}
\begin{cases}
-\Delta_{\gamma} u = -\Phi_{k+1}'(u) & \text{in } \Omega, \\
u = 0 & \text{on } \partial\Omega,
\end{cases}
\end{equation*}
a similar argument implies that $u_{0,k} \le u_{0,k+1}$ a.e. in $\om$. On the other hand, for every $\varepsilon > 0$, there exists
$\tilde{k} > \varepsilon^{-\delta}$ such that
\begin{equation*}
-\Delta_{\gamma} (u_{0,\tilde{k}} + \varepsilon)
=
-\Phi_{\tilde{k}}'(u_{0,\tilde{k}} + \varepsilon - \varepsilon)
\ge
-\Phi_k'(u_{0,\tilde{k}} + \varepsilon).
\end{equation*}
Therefore, $u_{0,\tilde{k}} + \varepsilon$ is a supersolution of $-\Delta_{\gamma} u = -\Phi_k'(u)$. By a weak comparison principle type argument, it is easy to obtain $u_{0,k} \le u_{0,\tilde{k}} + \varepsilon$. Consequently, the sequence $\{u_{0,k}\}$ is a Cauchy sequence in $L^\infty(\Omega)$. Hence, $\{u_{0,k}\}$ converges uniformly as $k \to \infty$ to some function $u_0 \in L^\infty(\Omega)$. Moreover, since $\varphi \le u_0 \le \psi$, we deduce that $u_0^{-\delta} \in L^\infty_{loc}(\Omega)$. Given $\varepsilon > 0$, testing the equation satisfied by $u_{0,k}$ with $(u_{0,k}-\varepsilon)^+$, we obtain
\begin{equation*}
\begin{aligned}
&\int_\Omega
\bigl| \gradgr (u_{0,k}-\varepsilon)^+ \bigr|^2 \, dz \\
&\quad =
-
\int_\Omega
\Phi_k'(u_{0,k}) (u_{0,k}-\varepsilon)^+ \, dz
\le
\varepsilon^{-\delta}
\int_\Omega
(u_{0,k}-\varepsilon)^+ \, dz
\le
\varepsilon^{-\delta}
\int_\Omega
(\psi-\varepsilon)^+ \, dz
< \infty .
\end{aligned}
\end{equation*}
It follows that the sequence $(u_{0,k}-\varepsilon)^+$ is bounded in $\spacegr$ as $k \to \infty$. Hence, by weak convegence, we conclude that $(u_0 - \varepsilon)^+ \in \spacegr$ for all $\varepsilon > 0$. 

Let $K \Subset \Omega$ be any compact set. Then $\varphi \ge 2\varepsilon_0$ on $K$ for some $\varepsilon_0 > 0$. 
Therefore, on $K$ we have $(u_0 - 2\varepsilon_0)^+ = u_0 - 2\varepsilon_0 \in \spacegr$, and hence
\begin{equation*}
\int_K
|\gradgr u_0|^2 \, dz
=
\int_K
\bigl| \gradgr (u_0 - 2\varepsilon_0) \bigr|^2 \, dz
< \infty .
\end{equation*}
For any $\phi \in C_c^\infty(K)$, we have
\begin{equation*}
\int_K
\gradgr (u_{0,k}-\varepsilon)^+ \cdot \gradgr \phi \, dz
\to
\int_K
\gradgr (u_0-\varepsilon)^+ \cdot \gradgr \phi \, dz
\end{equation*}
as $k \to \infty$. Therefore,
\begin{equation*}
\int_K
\gradgr (u_{0,k}-\varepsilon) \cdot \gradgr \phi \, dz
\to
\int_K
\gradgr (u_0-\varepsilon) \cdot \gradgr \phi \, dz ,
\end{equation*}
and equivalently,
\begin{equation*}
\int_K
\gradgr u_{0,k} \cdot \gradgr \phi \, dz
\to
\int_K
\gradgr u_0 \cdot \gradgr \phi \, dz .
\end{equation*}
Hence, we obtain weak convergence of the gradients on compact subsets. Consequently, we deduce that $u_0 \in H_{\gamma,loc}^{1}(\om)$, and $u_0 \leq 0$ on $\partial \om$ in the sense of \eqref{def:boundary_def}. Then, equation~\eqref{eq:regularized-problem} implies that
\begin{equation*}
-\Delta_{\gamma} u_0 = u_0^{-\delta}
\end{equation*}
in the sense of distributions. The uniqueness of $u_0$ follows from Lemma \ref{lem:comparison}.
\end{proof}


\section{Brezis-Nirenberg problem for a fully singular problem}\label{sec:brezis_nirenberg_grushin}
Finally, in this section, we study the existence and multiplicity problem for the following equation
where $\om$ is a bounded domain, $\delta>0,$ $\la > 0$ and $p \geq 1$.
To handle the singular term in the problem \eqref{eq:singular_semilinear_problem}, we consider the following translated problem 
\begin{equation}\tag{$\mathcal{T}$}
\label{eq:shifted_singular_problem}
\begin{cases}
-\Gr u + u_{0}^{-\delta} - (u + u_{0})^{-\delta} = \lambda (u + u_{0})^{p} & \text{in } \Omega, \\
u = 0 & \text{on } \partial\Omega,
\end{cases}
\end{equation}
where $u_{0}$ satisfies \eqref{eq:singular_grushin}. 
We mainly use the theorems developed in the previous sections to prove the existence and multiplicity results.


For notational convenience we define $g: \om \times \real \to [-\infty, \infty]$ by 
\begin{equation*}
g(z,s) =
\begin{cases}
u_{0}(z)^{-\delta} - \bigl( s + u_{0}(z) \bigr)^{-\delta},
& \text{for } (z,s) \in \Omega \times \mathbb{R} \text{ with } s + u_{0}(z) > 0, \\[0.6ex]
-\infty, & \text{otherwise},
\end{cases}
\end{equation*}
and its primitive as
\begin{equation*}
G(z,s) = \displaystyle \int_{0}^{s} g(z,\tau)\, d\tau,
\qquad \text{for } (z,s) \in \Omega \times \mathbb{R}.
\end{equation*}
Clearly, $g$ and $G$ satisfies assumption $(A2)$ defined in Section \ref{sec:brezis_nirenberg_general}.

We define the functional $I : \spacegr \to (-\infty,\infty]$ by
\begin{equation*}
I(u) =
\begin{cases}
\displaystyle
\frac{1}{2} \int_{\Omega} |\gradgr u|^{2}\, dz
+ \int_{\Omega} G(z,u)\, dz
- \frac{\lambda}{p+1} \int_{\Omega} |u + u_{0}|^{p+1}\, dz,
& \text{if } G(z,u) \in L^{1}(\Omega) \\ & \quad  \text{ and } u \in L^{p+1}(\Omega), \\[1.2ex]
\infty, & \text{otherwise}.
\end{cases}
\end{equation*}
We note that $u \in L^{p+1}(\om)$ is required only when we are in the supercritical case only. For every subset $K \subset \spacegr \cap L^{p+1}(\Omega)$ and for every $\lambda>0$, we also define the constrained functional $I_K : \spacegr \to (-\infty,\infty]$ by
\begin{equation*}
I_K(u) =
\begin{cases}
I(u), & \text{if } u \in K \text{ and } G(z,u) \in L^{1}(\Omega), \\
\infty, & \text{otherwise}.
\end{cases}
\end{equation*}
for every $u \in \spacegr$.

We set
\begin{equation*}
\Lambda
= \sup \bigl\{ \lambda>0 : \text{there exists a positive weak solution of } \eqref{eq:shifted_singular_problem}_{\lambda}
\text{ which belongs to } L^{\infty}(\Omega) \bigr\}.
\end{equation*}
We note that, in the subcritical and critical cases, one has
\begin{equation*}
\Lambda
= \sup \bigl\{ \lambda>0 : \text{there exists a positive weak solution of } \eqref{eq:shifted_singular_problem}_{\lambda} \bigr\}.
\end{equation*}

First, note that, in view of Lemma $4$ of \cite{hirano_concave_convex}, the functions $g$ and $G$ satisfies the following properties
\begin{lemma}
\label{lem:properties_of_G}
For each $z \in \Omega$, the following properties hold:
\begin{enumerate}
  \item[(i)]
  For every $r \ge 1$ and $s \ge 0$, $ G(z, r s) \le r^{2} \, G(z, s)$.

  \item[(ii)]
  For every $s,t$ with $s \ge t > -u_1(z)$ where $u_1$ satisfies \eqref{eq:one_grushin}, we have $ G(z, s) - g(z,t) - \frac{ g(z, s) + g(z,t) }{2}\, (s - t) \ge 0$.

  \item[(iii)]
  For every $s \ge 0$, $ G(z, s) - \frac{1}{2} \, g(z, s) \, s \ge 0$.
\end{enumerate}
\end{lemma}

Next, we show that $\Lambda>0$ and that, for every $\lambda \in (0,\Lambda)$, there exists a positive weak solution of $\eqref{eq:shifted_singular_problem}_{\lambda}$.

\begin{lemma}
The following statements hold:
\begin{enumerate}
  \item[(i)] For every $\lambda > 0$, the zero function is a strict subsolution of \eqref{eq:shifted_singular_problem}.
  
  \item[(ii)] For all sufficiently small $\lambda > 0$, the function $u_1$ is a strict supersolution of \eqref{eq:shifted_singular_problem} where $u_1$ satisfies \eqref{eq:one_grushin}.
  
  \item[(iii)] For any $\lambda, \mu$ with $\mu > \lambda > 0$ and for every positive weak solution $z$ of \eqref{eq:shifted_singular_problem} corresponding to $\mu$ and belonging to $L^{\infty}(\Omega)$, the function $z$ is a strict supersolution of \eqref{eq:shifted_singular_problem} corresponding to $\lambda$.
\end{enumerate}
\end{lemma}
\begin{proof}
$(i)$ and $(iii)$ follows trivially. For $(ii)$, note that $u_1, u_{0} \in L^{\infty}(\om)$, then there exists $\la_0$ sufficiently small such that for all $0 < \la < \la_0$ we have $1 > \la (u_1 + u_{0})^p$ for every $z \in \om$. Then for every $w \in C_c^{\infty}(\om)$, we get
\begin{equation*}
    \int_{\om} ( \gradgr u_1 \cdot \gradgr w + g(z,u_1)w - \la (u_1 + u_{0})^pw)dz \geq \int_{\om} (1- \la(u_1 + u_{0})^p)w dz > 0.
\end{equation*}
This proves $(ii)$.
\end{proof}

\begin{lemma}\label{lem:existence_critical_firstsoln}
We have $\Lambda>0$. Moreover, for every $\lambda\in(0,\Lambda)$ there exists a positive weak solution $u_\lambda$ of \eqref{eq:shifted_singular_problem} such that
\[
u_\lambda \in C_{loc}^{0,\al}(\Omega)\cap L^\infty(\Omega), \qquad I(u_\lambda)<0.
\]
In addition, $u_\lambda$ is a local minimizer for $I_{K_0}$ in the subcritical and critical regimes, and $u_\lambda$ is a local minimizer for $I_{K_0^v}$ in the supercritical regime, where $v$ is any function in $\spacegr\cap L^\infty(\Omega)$ satisfying $v\ge u_\lambda$. Furthermore, in the subcritical case there exists $\rho_0>0$ such that
\begin{equation*}
I_{K_0}(u_\lambda) < \inf\bigl\{ I_{K_0}(v) : v\in K_0,\ \|(v-\varphi_2)^+\| = \rho \bigr\}
\end{equation*}
for all $\rho\in(0,\rho_0]$.
\end{lemma}

\begin{proof}
By Lemma~5(i), the zero function is a strict subsolution of \eqref{eq:shifted_singular_problem} for every $\lambda>0$. Fix $\lambda>0$ such that there exists a positive strict supersolution $\varphi_2$ of \eqref{eq:shifted_singular_problem} with $\varphi_2\in L^\infty(\Omega)$. It is straightforward to verify that the hypotheses of Theorem \ref{thm:existence_between_sub_super} (in the subcritical and critical cases) or of Theorem \ref{thm:local_minimizer_from_supersolution_supercritical} (in the supercritical case).
\end{proof}

\begin{lemma}
\label{lem:weak_solution_identity}
If $u$ is nonnegative weak solution of \eqref{eq:shifted_singular_problem}, then it satisfies $g(z,u)w \in L^1(\Omega)$ and
\begin{equation*}
\int_{\Omega} \bigl( \gradgr u \cdot \gradgr w + g(z,u)w - \lambda (u+u_{0})^p w \bigr)\, dz = 0
\qquad \text{for every } w \in \spacegr.
\end{equation*}
In the supercritical regime, we also need to assume that $u \in L^{\infty}(\om)$.
\end{lemma}
\begin{proof}
Since every $w \in \spacegr$ can be written as $w = w^+ - w^-$. It is enough to consider $w \geq 0$. By the Lemma \ref{lem:density_grushin_space}, there exists an increasing sequence $w_k \in \spacegr \cap L^{\infty}(\om)$. Then by definition of $u$, we have 
\begin{equation*}
\int_{\om} g(z,u) w_k \, dz  =  \int_{\Omega} \bigl( - \gradgr u \cdot \gradgr w_k + \lambda (u+u_{0})^p w_k \bigr)\, dz .
\end{equation*}
Passing the limit with the help of Monotone convergence theorem, we have our conclusion. 
\end{proof}
Let $u$ denotes the first solution obtained in Lemma \ref{lem:existence_critical_firstsoln}, then we have the following proposition
\begin{proposition}
\label{prop:CPS_condition}
The functional $I_{K_u}$ satisfies $(\mathrm{CPS})_c$ for every $c \in \mathbb{R}$ in the subcritical case, and for each $c < I_{K_u}(u) + \frac{S^{Q/2}}{Q \lambda^{(Q-2)/2}}$ in the critical case.
\end{proposition}

\begin{proof}
Choose a sequence $\{v_k\} \subset D(I_{K_u})$ such that
\[
I_{K_u}(v_k) \to c
\quad \text{and} \quad
(1 + \|v_k\|)\, \|\partial^- I_{K_u}(v_k)\| \to 0 
\]
as $k \to \infty$. For every $k \in \mathbb{N}$, by definition there exists $\alpha_k \in \partial^- I_{K_u}(v_k)$ with $\|\alpha_k\| = \|\partial^- I_{K_u}(v_k)\|$.
By Lemma~\ref{lem:variational_inequality_characterization}, for every $k \in \mathbb{N}$ and every $w \in D(I_{K_u})$, we have $g(\cdot, v_k)(w - v_k) \in L^1(\Omega)$ and
\begin{equation}
\label{eq:ps_inequality_4_3}
\begin{aligned}
\langle \alpha_k, w - v_k \rangle
&\le \int_{\Omega} \gradgr v_k \cdot \gradgr (w - v_k)
   + \int_{\Omega} g(z, v_k)\, (w - v_k) - \lambda \int_{\Omega} (v_k + u_{0})^p (w - v_k).
\end{aligned}
\end{equation}
From Lemma~\ref{lem:properties_of_G}(i) and the assumption $G(\cdot, v_k) \in L^1(\Omega)$, we obtain $G(\cdot, 2v_k) \in L^1(\Omega)$, which implies $2v_k \in D(I_{K_u})$. Substituting $w = 2v_k$ into \eqref{eq:ps_inequality_4_3}, we get
\begin{equation}\label{eq:ineq1_cpc}
\langle \alpha_k, v_k \rangle
\le \int_{\Omega} |\gradgr v_k|^2\, dz
+ \int_{\Omega} g(z, v_k) v_k\, dz
- \lambda \int_{\Omega} (v_k + u_{0})^p v_k\, dz .
\end{equation}
Also, note that 
\begin{equation*}
\lim_{t \to \infty} \frac{\int_{\om} (t + u_0)^{p+1} dz}{\int_{\om} (t+u_0)^p t \, dz} = 1.
\end{equation*}
This gives for any $\varepsilon > 0$, there exists $M_{\varepsilon} > 0$ such that 
\begin{equation*}
    \int_{\om} (v_k + u_0)^{p+1} \, dz \leq (1 + \varepsilon) \int_{\om} (v_k + u_0)^{p} v_k \, dz + M_{\varepsilon}.
\end{equation*}
Thus using above, \eqref{eq:ineq1_cpc}, and the Lemma \ref{lem:properties_of_G}(iii), we deduce
\begin{align*}
    c  + 1 & \geq I(v_k)  \\
    & \geq  \frac{1}{2} \int_{\Omega} |\gradgr v_k|^2\, dz
+ \int_{\Omega} G(z, v_k)\, dz
- \frac{\lambda}{p+1} \int_{\Omega} (v_k + u_{0})^{p+1}\, dz \\
 & \geq \frac{1}{2} \int_{\Omega} |\gradgr v_k|^2\, dz
+ \int_{\Omega} G(z, v_k)\, dz - \frac{\lambda}{p+1} M_\varepsilon \\
&\quad + \left(\frac{1+\varepsilon}{p+1}\right) \left(
\langle \alpha_k, v_k \rangle
- \int_{\Omega} |\gradgr v_k|^2\, dz
- \int_{\Omega} g(z, v_k) v_k\, dz
\right)\\
& \geq \left( \frac{1}{2} - \left(\frac{1+\varepsilon}{p+1}\right) \right) \int_{\om} |\gradgr v_k|^2\, dz 
+ \left( \frac{1}{2} - \left(\frac{1+\varepsilon}{p+1} \right) \right) \int_{\Omega} g(z, v_k) v_k\, dz \\
& \quad  + (1 + \varepsilon) \frac{\langle \alpha_k, v_k \rangle}{p+1}.
\end{align*}
This yields that the sequence $\{v_k\}$ is bounded in $\spacegr$. Hence, up to a subsequence, we may assume that $v_k \rightharpoonup v$ weakly in $\spacegr$ and $v_k(z) \to v(z)$ a.e. in $\Omega$, and that
\[
\int_{\Omega} |\gradgr (v_k - v)|^2\, dz \to a^2,
\qquad
\int_{\Omega} |v_k - v|^{p+1}\, dz \to b^{p+1}.
\]
Moreover, convexity of $G$ and \eqref{eq:ps_inequality_4_3} with $w = v$ yields
\begin{equation*}
\begin{aligned}
&\int_{\Omega} G(z, v)\, dz
\ge \int_{\Omega} G(z, v_k)\, dz
+ \int_{\Omega} g(z, v_k) (v - v_k)\, dz \\
&\ge \int_{\Omega} G(z, v_k)\, dz
+ \int_{\Omega} \gradgr v_k \cdot \gradgr (v_k - v)\, dz - \lambda \int_{\Omega} (v_k + u_{0})^p (v_k - v)\, dz
- \langle \alpha_k, v_k - v \rangle.
\end{aligned}
\end{equation*}
From above and Proposition \ref{ineq:imp_ineq} we deduce that
\begin{equation*}
\int_{\Omega} G(z, v)\, dz
\ge \int_{\Omega} G(z, v)\, dz + a^2 - \lambda b^{p+1}.
\end{equation*}
This yields $\lambda b^{p+1} \ge a^2$. Since $b=0$ in the subcritical case, we have $a = 0$. Hence $I_{K_u}$ satisfies $(\mathrm{CPS})_c$.

We now consider the critical case. Since $u$ is a positive weak solution, by Lemma \ref{lem:weak_solution_identity} we have
\begin{equation}
\label{eq:critical_identity_4_4}
\int_{\Omega} \bigl(
\gradgr u \cdot \gradgr (v_k - u)
+ g(z, u) (v_k - u)
- \lambda (u + u_{0})^{2_{\gamma}^* - 1} (v_k - u)
\bigr)\, dz = 0 .
\end{equation}

From the facts that $G(\cdot, v_k), G(\cdot, 2v_k) \in L^1(\Omega)$ and $u \le 2v_k - u \le 2v_k$, we infer that $2v_k - u \in D(I_{K_u})$. Using \eqref{eq:critical_identity_4_4}, the inequality \eqref{eq:ps_inequality_4_3} with $w = 2v_k - u$, and Lemma \eqref{lem:properties_of_G}(ii), we obtain
\begin{equation*}
\begin{aligned}
I_{K_u}(v_k) - I_{K_u}(u)
&= \frac{1}{2} \int_{\Omega} |\gradgr v_k|^2\, dz
+ \int_{\Omega} G(z, v_k)\, dz
- \frac{\lambda}{2_{\gamma}^*} \int_{\Omega} |v_k + u_{0}|^{2_{\gamma}^*}\, dz \\
&\quad - \frac{1}{2} \int_{\Omega} |\gradgr u|^2\, dz
- \int_{\Omega} G(z, u)\, dz
+ \frac{\lambda}{2_{\gamma}^*} \int_{\Omega} |u + u_{0}|^{2_{\gamma}^*}\, dz \\
&\ge \int_{\Omega} \Bigl(
G(z, v_k) - G(z, u)
- \frac{1}{2} g(z, v_k)(v_k - u)
- \frac{1}{2} g(z, u)(v_k - u)
\Bigr)\, dz \\
&\quad + \lambda \int_{\Omega} \Bigl(
\frac{1}{2} |v_k + u_{0}|^{2_{\gamma}^* - 1} (v_k - u)
- \frac{1}{2_{\gamma}^*} |v_k + u_{0}|^{2_{\gamma}^*} \\
&\qquad\qquad\qquad
+ \frac{1}{2} |u + u_{0}|^{2_{\gamma}^* - 1} (v_k - u)
+ \frac{1}{2_{\gamma}^*} |u + u_{0}|^{2_{\gamma}^*}
\Bigr)\, dz
+ \frac{1}{2} \langle \alpha_k, v_k - u \rangle \\
&\ge \lambda \int_{\Omega} \Bigl(
\frac{1}{2} |v_k + u_{0}|^{2_{\gamma}^* - 1} (v_k - u)
- \frac{1}{2_{\gamma}^*} |v_k + u_{0}|^{2_{\gamma}^*} \\
&\qquad\qquad\qquad
+ \frac{1}{2} |u + u_{0}|^{2_{\gamma}^* - 1} (v_k - u)
+ \frac{1}{2_{\gamma}^*} |u + u_{0}|^{2_{\gamma}^*}
\Bigr)\, dz
+ \frac{1}{2} \langle \alpha_k, v_k - u \rangle .
\end{aligned}
\end{equation*}

Using Brezis-Lieb Lemma and the convexity of the map $\tau \mapsto |\tau + u_{0}(x)|^{2_{\gamma}^* - 1}$, we infer that
\begin{equation*}
\begin{aligned}
c - I_{K_u}(u)
&\ge \frac{\lambda b^{2_{\gamma}^*}}{Q}
+ \lambda \int_{\Omega} \Bigl(
\frac{1}{2} |v + u_{0}|^{2_{\gamma}^* - 1} (v - u)
- \frac{1}{2_{\gamma}^*} |v + u_{0}|^{2_{\gamma}^*} \\
&\qquad\qquad\qquad\qquad
+ \frac{1}{2} |u + u_{0}|^{2_{\gamma}^* - 1} (v - u)
+ \frac{1}{2_{\gamma}^*} |u + u_{0}|^{2_{\gamma}^*}
\Bigr)\, dz \\
&= \frac{\lambda b^{2_{\gamma}^*}}{Q}
+ \lambda \int_{\Omega} \left(
\frac{|v + u_{0}|^{2_{\gamma}^* - 1} + |u + u_{0}|^{2_{\gamma}^* - 1}}{2} (v - u)
- \int_{u}^{v} |\tau + u_{0}|^{2_{\gamma}^* - 1}\, d\tau
\right) dz \\
&\ge \frac{\lambda b^{2_{\gamma}^*}}{Q}.
\end{aligned}
\end{equation*}

If $a > 0$, then $\lambda b^{2_{\gamma}^*} \ge a^2$ and $a^2 \ge S b^2$ yield
\[
\frac{\lambda b^{2_{\gamma}^*}}{Q} \ge \frac{S^{Q/2}}{Q \lambda^{(Q-2)/2}},
\]
and hence we obtain
\[
c \ge I_{K_u}(u) + \frac{S^{Q/2}}{Q \lambda^{(Q-2)/2}},
\]
which contradicts the assumption $c < I_{K_u}(u) + \dfrac{S^{Q/2}}{Q \lambda^{(Q-2)/2}}$.
Therefore $a = 0$, and consequently $\{v_k\}$ converges to $v$ strongly in $\spacegr$.
Hence $I_{K_u}$ satisfies $(\mathrm{CPS})_c$. 
\end{proof}

Next, we show that in the critical case there exists a Cerami sequence at a level below the threshold established in the preceding lemma.
\begin{lemma}\label{lem:critical-case-psi}
In the critical regime, one can construct a nonnegative function $\Psi \in \spacegr$ such that the following bound holds:
\begin{equation*}
\sup\left\{ I_{K_u}(u + t\Psi) : t \ge 0 \right\}
<
I_{K_u}(u) + \frac{S^{Q/2}}{Q\,\lambda^{(Q-2)/2}}.
\end{equation*}
\end{lemma}
\begin{proof}
We take $\Psi_{\varepsilon} = u_{\varepsilon} = \zeta v_{\varepsilon}$ where $u_{\varepsilon}$ defined as in \eqref{eq:grushin_bubbles}. Fix $q$ such that $1 < q < \min\{2,\, Q/(Q-2)\}$. Set $\eta_\zeta = \sup\{ |x| : x \in \operatorname{supp}\zeta \}$. Then there exist $m, M > 0$ such that $m \leq u(z) \leq M$ for all $z \in \eta_{\zeta}$.

 
Using the upper bound in \eqref{eq:asymptotic_estimate_fundamental}, we obtain the estimate
\begin{equation}\label{eq:psi-delta-bound}
\begin{aligned}
\int_{\Omega} |\Psi_\varepsilon|^q \, dz \leq \varepsilon^{q(2-Q)/2} \int_{B_{\eta_{\zeta}}}d\left( \frac{z}{\varepsilon} \right)^{q(2-Q)} dz
\leq C_2 \varepsilon^{q\frac{2-Q}{2} + Q}\int_{B_{\eta_{\zeta}/\varepsilon}} d(z)^{q(2-Q)} dz
\le C_3 \, \varepsilon^{\frac{(Q-2)q}{2}}. 
\end{aligned}
\end{equation}
In view of representation formula in Section $3.1$ and Remark $2.3$ of \cite{ambrosio_representation_formula}, there exists $C>0$ such that $U(z) \geq C$ for all $d(z) \leq 1$. This yields
\begin{equation}\label{eq:psi_subcritical_lower}
\int_{|x| \le \varepsilon} |\Psi_\varepsilon|^{2_{\gamma}^* - 1} \, dz 
\geq \varepsilon^{-\frac{(Q+2)}{2}} \int_{|z|\leq \varepsilon} |U(z/\varepsilon)|^{\frac{Q+2}{Q-2}} dz \geq C \varepsilon^{-\frac{(Q+2)}{2}} \int_{|z|\leq \varepsilon} dz
\ge C_4 \, \varepsilon^{\frac{Q-2}{2}}.
\end{equation}
Next, observe that
\begin{equation*}
\begin{aligned}
G(z, r+s) - G(z, r) - g(z, r)s
&= \int_r^{r+s} \big( g(z,t) - g(z, r) \big) \, dt \\
&= \int_r^{r+s} \left( (r + \tilde u(z))^{-\delta} - (t + \tilde u(z))^{-\delta} \right) dt \\
&\le \int_r^{r+s} \left( r^{-\delta} - t^{-\delta} \right) dt .
\end{aligned}
\end{equation*}
From this estimate and mean value theorem, there exists a constant $\alpha > 0$ such that
\begin{equation}\label{eq:G-alpha-estimate}
G(x, r+s) - G(x, r) - g(x, r)s \le \alpha s^q,
\quad \text{for all } z \in \Omega,\; r \ge m >0,\; s \ge 0 .
\end{equation}
Moreover, one readily checks that
\begin{equation}\label{eq:power-inequality-basic}
\frac{(r+s)^{2_{\gamma}^*}}{ 2_{\gamma}^*} - \frac{r^{2_{\gamma}^*}}{2_{\gamma}^*} - r^{2_{\gamma}^* - 1} s \ge \frac{s^{2_{\gamma}^*}}{2_{\gamma}^*},
\quad \text{for all } r, s \ge 0,
\end{equation}
and by Taylor's expansion and mean value theorem, there exists a constant $\beta > 0$ such that
\begin{equation}\label{eq:power-inequality-refined}
\frac{(r+s)^{2_{\gamma}^*}}{2_{\gamma}^*} - \frac{r^{2_{\gamma}^*}}{2_{\gamma}^*} - r^{2_{\gamma}^* - 1} s
\ge \frac{s^{2_{\gamma}^*}}{2_{\gamma}^*} + \beta \frac{r s^{2_{\gamma}^* - 1}}{2_{\gamma}^* - 1},
\quad \text{for all } 0 \le r \le M \text{ and } s \ge 1 .
\end{equation}
Since $u$ is a positive weak solution of the problem \eqref{eq:shifted_singular_problem}, we have
\begin{equation}\label{eq:IK-difference-start}
\begin{aligned}
I_{K_u}&(u + t\Psi_\varepsilon) - I_{K_u}(u) \\
&= I_{K_u}(u + t\Psi_\varepsilon) - I_{K_u}(u)
- t \int_{\Omega} \big( \nabla u \cdot \nabla \Psi_\varepsilon + g(z,u)\Psi_\varepsilon - \lambda (u + u_{0})^{2_{\gamma}^*-1} \Psi_\varepsilon \big)\, dz \\
&= \frac{t^2}{2} \int_{\Omega} |\nabla \Psi_\varepsilon|^2 \, dz
+ \int_{\Omega} \big( G(z, u + t\Psi_\varepsilon) - G(z, u) - g(z,u)(t\Psi_\varepsilon) \big)\, dz \\
&\quad - \lambda \int_{\Omega} \left(
\frac{1}{2_{\gamma}^*} |u + t\Psi_\varepsilon + u_{0}|^{2_{\gamma}^*}
- \frac{1}{2_{\gamma}^*} |u + u_{0}|^{2_{\gamma}^*}
- (u + u_{0})^{2_{\gamma}^*-1} (t\Psi_\varepsilon)
\right) dz .
\end{aligned}
\end{equation}
Using \eqref{eq:power-inequality-basic}, \eqref{eq:G-alpha-estimate}, \eqref{eq:psi-delta-bound} and estimates $(i)$ and $(ii)$ of Lemma \ref{lem:blow_up_grushin} in \eqref{eq:IK-difference-start}, we infer that
\begin{equation}\label{eq:IK-upper-bound-small-t}
\begin{aligned}
I_{K_u}(u + t\Psi_\varepsilon) - I_{K_u}(u)
&\le \frac{t^2}{2} \int_{\Omega} |\nabla \Psi_\varepsilon|^2 \, dz
- \frac{\lambda t^{2_{\gamma}^*}}{2_{\gamma}^*} \int_{\Omega} |\Psi_\varepsilon|^{2_{\gamma}^*} \, dz
+ \alpha t^q \int_{\Omega} |\Psi_\varepsilon|^q \, dz \\
&\le \frac{t^2}{2} \big( S^{Q/2} + C_1 \varepsilon^{Q-2} \big)
- \frac{\lambda t^{2_{\gamma}^*}}{2_{\gamma}^*} \big( S^{Q/2} - C_2 \varepsilon^{Q} \big)
+ \alpha C_3 t^q \varepsilon^{\frac{(Q-2)q}{2}},
\end{aligned}
\end{equation}
for all $0 \le t < \lambda^{-4/(Q-2)}/2$.

Next, since we may assume that $t\Psi_\varepsilon(x) \ge 1$ for every $t \ge \lambda^{-4/(Q-2)}/2$ and $|z| \le \varepsilon$, using \eqref{eq:power-inequality-refined}, \eqref{eq:G-alpha-estimate}, \eqref{eq:psi-delta-bound}, \eqref{eq:psi_subcritical_lower} and estimates $(i)$ and $(ii)$ of Lemma \ref{lem:blow_up_grushin} in \eqref{eq:IK-difference-start}, we obtain
\begin{equation}\label{eq:IK-upper-bound-large-t}
\begin{aligned}
I_{K_u}&(u + t\Psi_\varepsilon) - I_{K_u}(u)\\
&\le \frac{t^2}{2} \int_{\Omega} |\nabla \Psi_\varepsilon|^2 \, dz
- \frac{\lambda t^{2_{\gamma}^*}}{2_{\gamma}^*} \int_{\Omega} |\Psi_\varepsilon|^{2_{\gamma}^*} \, dz
- \frac{\lambda \beta c_4 t^{2_{\gamma}^*-1}}{2_{\gamma}^* - 1} \int_{|z| \le \varepsilon} |\Psi_\varepsilon|^{2_{\gamma}^*-1} \, dz
+ \alpha t^q \int_{\Omega} |\Psi_\varepsilon|^q \, dz \\
&\le \frac{t^2}{2} \big( S^{Q/2} + C_1 \varepsilon^{Q-2} \big)
- \frac{\lambda t^{2_{\gamma}^*}}{2_{\gamma}^*} \big( S^{Q/2} - C_2 \varepsilon^{Q} \big)
- \frac{\lambda \beta C_4 t^{2_{\gamma}^*-1}}{2_{\gamma}^* - 1} \varepsilon^{\frac{Q-2}{2}}
+ \alpha C_3 t^q \varepsilon^{\frac{(Q-2)q}{2}},
\end{aligned}
\end{equation}
for all $t \ge \lambda^{-4/(Q-2)}/2$.

Define the function $j_\varepsilon : [0, \infty) \to \mathbb{R}$ by the right-hand sides of
\eqref{eq:IK-upper-bound-small-t} on $[0, \lambda^{-4/(Q-2)}/2)$ and of \eqref{eq:IK-upper-bound-large-t} on
$[\lambda^{-4/(Q-2)}/2, \infty)$. A direct computation shows that $j_\varepsilon$ attains its maximum at
\[
t = \lambda^{-4/(Q-2)} \left( 1 - \frac{\beta c_4 \varepsilon^{(Q-2)/2}}{(2_{\gamma}^* - 2) S^{Q/2}} \right)
+ o\big( \varepsilon^{(Q-2)/2} \big),
\]
and therefore
\begin{equation*}
\sup_{t \ge 0} \big( I_{K_u}(u + t\Psi_\varepsilon) - I_{K_u}(u) \big)
\le
\frac{S^{Q/2}}{Q \lambda^{(Q-2)/2}}
- \frac{\beta c_4}{(2_{\gamma}^* - 1)\lambda^{(Q-2)/4}} \varepsilon^{\frac{Q-2}{2}}
+ o\!\left( \varepsilon^{\frac{Q-2}{2}} \right)
<
\frac{S^{Q/2}}{N \lambda^{(Q-2)/2}} .
\end{equation*}
This completes the proof.
\end{proof}

\begin{proposition}
In both the subcritical and the critical regimes, for every $\lambda \in (0,\Lambda)$, problem \eqref{eq:shifted_singular_problem} admits a second positive weak solution.
\end{proposition}

\begin{proof}
In the critical case, let $\Psi$ be as constructed in Lemma~\ref{lem:critical-case-psi}. Define
\begin{equation*}
w =
\begin{cases}
u, & \text{in the subcritical case}, \\
\Psi, & \text{in the critical case}.
\end{cases}
\end{equation*}
Since $u$ is a local minimizer of $I_{K_u}$, there exists $\rho > 0$ such that $I_{K_u}(v) \ge I_{K_u}(u)$ for every $ v \in K_u $ with $\|v-u\| \le \rho$. Moreover, using Lemma \ref{lem:properties_of_G}(i), we know that
\begin{equation*}
I_{K_u}(u + t w) \to -\infty
\quad \text{as } t \to \infty .
\end{equation*}
Hence, one can select $t_0 > \rho/\|w\|$ such that
\begin{equation*}
I_{K_u}(u + t_0 w) \le I_{K_u}(u).
\end{equation*}
In order to apply Theorem \ref{thm:Linking_theorem}, we introduce the class of admissible paths
\begin{equation*}
\Phi
=
\left\{
\varphi \in C([0,1], D(I_{K_u})) :
\varphi(0)=u,\;
\varphi(1)=u+t_0w
\right\},
\end{equation*}
and define
\begin{equation*}
A
=
\left\{
v \in D(I_{K_u}) : \|v-u\|=\rho
\right\},
\qquad
c
=
\inf_{\varphi \in \Phi}
\sup_{0 \le s \le 1}
I_{K_u}(\varphi(s)).
\end{equation*}
From the choice of $t_0$, we note that $A \cap \varphi(1) = \varnothing $. By the Lemma \ref{lem:critical-case-psi} together with Proposition \ref{prop:CPS_condition}, the functional $I_{K_u}$ satisfies the condition $(\mathrm{CPS})_c$. 
If $c > I_{K_u}(u)$, then by Theorem \ref{thm:Linking_theorem} we are done. 
If $c = I_{K_u}(u)$, we observe that $u \notin A$, $u+t_0w \notin A$, and since
\begin{equation*}
\inf I_{K_u}(A) \ge c = I_{K_u}(u) \ge I_{K_u}(u+tw),
\end{equation*}
and for every $\varphi \in \Phi$ there exists $s \in [0,1]$ such that $\|\varphi(s)-u\| = \rho $, and thus $c \geq I_{K_u}(A)$.
Therefore, by Theorem \ref{thm:Linking_theorem}, there exists $v \in A \cap D(I_{K_u})$ with $v \neq u$ such that
\begin{equation*}
I_{K_u}(v) = c,
\qquad
0 \in \partial^- I_{K_u}(v).
\end{equation*}
Finally, by Proposition \ref{prop:sub_super_solution_criterion}(i), the function $v$ is a positive weak solution of problem \eqref{eq:shifted_singular_problem}.
\end{proof}
Note that the next lemma follows in the similar way as of Lemma $8$ of \cite{hirano_fully_singular}.
\begin{lemma}
Let $v$ be a positive weak solution of \eqref{eq:shifted_singular_problem} such that $v \ge u$, $v \neq u$, and $v \in L^\infty(\Omega)$. Then
\begin{equation*}
I_{K_u}^v(v) > I_{K_u}^v(u).
\end{equation*}
\end{lemma}

\begin{lemma}
Let $w$ be a positive weak solution of \eqref{eq:shifted_singular_problem}$_\mu$ with $w \in L^\infty(\Omega)$ and $\mu \ge \lambda$. Then
\begin{equation*}
w > u \quad \text{in } \Omega .
\end{equation*}
\end{lemma}

\begin{proof}
Following the similar lines of Lemma $9$ of \cite{hirano_fully_singular}, we have $w \geq u$. 

Finally, observe that
\begin{equation*}
-\Gr  (w-u) + g(z,w) - g(w,u)
=
\mu (w+u_{0})^p - \lambda (u+u_{0})^p
> 0
\quad \text{in } \Omega .
\end{equation*}
Consider the function 
\begin{equation*}
    h(z) = 
    \begin{cases}
        \frac{g(z,w) - g(z,u)}{w-u} & \text{ if } w> u, \\
        0 & \text{ if } w = u. 
    \end{cases}
\end{equation*}
Then by mean value theorem we have $|h(z)| \leq \delta (u + u_{0})^{-\delta-1}$. But we know that on every compact set $K \subset \om$, there exists a constant $\omega_K$ such that $u_{0} \geq \omega_K$. This gives $|h(z)| \leq \delta \omega_K^{-\delta-1}$. Thus $h \in L_{loc}^{1}(\om)$. By the strong maximum principle, we conclude that $w>u$ in $\Omega$. This completes the proof.
\end{proof}

The following two lemmas tells the existence of the solution in the limiting case. The proof of these two lemmas are essentially the same as that of Lemma $10$ and $11$ of \cite{hirano_fully_singular}.
\begin{lemma}
$\Lambda < \infty$ .
\end{lemma}

\begin{lemma}
Even in the extremal case $\lambda = \Lambda$, problem \eqref{eq:shifted_singular_problem}$_\Lambda$ admits a positive weak solution belonging to $L^{p+1}(\Omega)$.
\end{lemma}




Next, we are interested in the regularity of the very weak solutions we obtained of \eqref{eq:shifted_singular_problem}. In the following two lemmas, we are interested in the uniform estimate of the solutions.

\begin{lemma}\label{lem:Lr-regularity}
In both the subcritical and the critical regimes, every nonnegative weak solution of \eqref{eq:shifted_singular_problem}$_\lambda$ belongs to $L^r(\Omega)$ for all $r\in[1,\infty)$.
\end{lemma}

\begin{proof}
The proof is the adaptation of the argument of \cite[Lemma B.3]{struwe_variational_methods}. Let $u$ be any nonnegative weak solution of \eqref{eq:shifted_singular_problem}$_\lambda$. 
Fix $\ell\ge0$ and take $\beta\in[1,\infty)$, we define $w = \min\{u^{\beta-1}, \ell\}$.
Then $uw,\,uw^2\in \spacegr$ for $\beta = 1$. Testing with $w$, for any $C_{0}>0$, we obtain
\begin{equation*}
\begin{aligned}
&\int_\Omega |\gradgr(uw)|^2\,dz
\le \beta \int_\Omega \gradgr u\cdot\gradgr(uw^2)\,dz  = \beta\int_\Omega\big(-g(z,u)+\lambda(u+u_{0})^p\big)uw^2\,dz \\
& \leq 2^{p-1}\beta\lambda\int_\Omega (u^p+u_{0}^p)uw^2\,dz \\
&\le 2^{p-1}\beta\lambda\int_\Omega u_{0}^p uw^2\,dz
+2^{p-1}\beta\lambda\left(
\int_{\{u\le C_{0}\}} u^{2\beta+p-1}\,dz
+\int_{\{ u>C_{0} \}} u^{p-1}u^2w^2\,dz
\right)\\
&\le 2^{p-1}\beta\lambda\Big(
\|u_{0}\|_\infty^p\,\|u\|_{2\beta-1}^{2\beta -1}
+|\Omega|C_{0}^{2\beta+p-1}
\Big)  +C\left(
\int_{\{u>C_{0}\}} u^{\frac{(p-1)Q}{2}}\,dz
\right)^{\!\frac{2}{Q}}
\int_\Omega |\gradgr(uw)|^2\,dz .
\end{aligned}
\end{equation*}
Choosing $C_{0}$ sufficiently large we may ensure $C\left(
\int_{\{ u>C_{0} \}} u^{\frac{(p-1)Q}{2}}\,dz
\right)^{\!\frac{2}{Q}}
\le \frac12$. Therefore,
\begin{equation*}
\int_{\{u^{\beta-1}\le \ell\}} |\gradgr u^\beta|^2\,dz
\le
\int_\Omega |\gradgr(uw)|^2\,dz
\le
2^p\beta\lambda
\Big(
\|u_{0}\|_\infty^p\,\|u\|_{2\beta-1}^{2\beta - 1}
+|\Omega|C_{0}^{2\beta+p-1}
\Big).
\end{equation*}

Letting $\ell\to\infty$ yields $u^\beta\in \spacegr$. By the Sobolev embedding we deduce that $u\in L^{\frac{2\beta Q}{Q-2}}(\Omega)$.
Iterating this estimate provides $u\in L^r(\Omega)$ for every $r\in[1,\infty)$, which completes the proof.
\end{proof}

\begin{lemma}\label{lem:L-infty-regularity}
In both the subcritical and critical cases, every nonnegative weak solution of \eqref{eq:shifted_singular_problem}$_\lambda$ belongs to $L^\infty(\Omega)$.
\end{lemma}

\begin{proof}
Let $u$ be a nonnegative weak solution. For any nonnegative test function $\psi\in C_c^\infty(\Omega)$, we have
\begin{equation*}
\int_{\Omega} \gradgr u \cdot \gradgr \psi \, dz
\le
\int_{\Omega} \lambda (u+u_{0})^{p}\psi \, dz .
\end{equation*}
Combining the $L^r$-regularity obtained in Lemma~\ref{lem:Lr-regularity} with Lemma \ref{lem:uniform_estimate_grushin}, we deduce that $u\in L^\infty(\Omega)$.
\end{proof}




\begin{lemma}\label{lem:truncation-property}
Let $r>0$, let $v\in L^{(r+1)/r}(\Omega)$ be a positive function, and let $u\in \spacegr\cap L^{r+1}(\Omega)$ be a positive weak solution of
\begin{equation*}
\begin{cases}
-\Gr u + g(z,u) = v & \text{in } \Omega,\\
u = 0 & \text{on } \partial\Omega .
\end{cases}
\end{equation*}
Then $(u+u_{0}-\varepsilon)^+ \in \spacegr$ for every $\varepsilon>0$. In particular, every positive weak solution $u$ of \eqref{eq:shifted_singular_problem} belonging to $L^{r+1}(\Omega)$ satisfies the same property.
\end{lemma}

\begin{proof}
Let $\varepsilon,\sigma>0$ and since $(u_{0} - \sigma)^+ \in \spacegr$, we have $w=\min\{u,\varepsilon-(u_{0}-\sigma)^+\}\in \spacegr$. Observe that
\begin{equation*}
0\le v(u-w)
= v\big(u+(u_{0}-\sigma)^+-\varepsilon\big)^+
\le vu+vu_{0},
\end{equation*}
and $vu+vu_{0}\in L^1(\Omega)$, an argument similar to that used in Lemma \ref{lem:weak_solution_identity} shows that $g(\cdot,u)(u-w)\in L^1(\Omega)$ and
\begin{equation}\label{eq:weak-identity}
\int_{\Omega}
\left(
\gradgr u\cdot\gradgr(u-w)
+ g(z,u)(u-w)
- v(u-w)
\right)\,dz
=0 .
\end{equation}
Furthermore, since
\begin{equation}\label{eq:singular_weak_eq}
\int_{\Omega}\gradgr u_{0}\cdot\gradgr\varphi\,dz
=
\int_{\Omega}u_{0}^{-\delta}\varphi\,dz
\quad
\text{for every } \varphi\in C_c^\infty(\Omega),\ \varphi\ge0.
\end{equation}
Taking $\varphi = \eta_{\varepsilon}(\bar{u}) \psi$ in \eqref{eq:singular_weak_eq} where $\psi \in C_{c}^{\infty}(\om)$, $\psi \geq 0$ and 
\begin{equation*}
    \eta_{\varepsilon}(t) = 
    \begin{cases}
        0, & t\leq \sigma \\
        \frac{t-\sigma}{\varepsilon}, & \sigma < t < \sigma + \varepsilon \\
        1, & t \geq \sigma + \varepsilon.
    \end{cases}
\end{equation*}
and passing to the limit, we obtain
\[
\int_{\Omega}\gradgr(u_{0}-\sigma)^+\cdot\gradgr\psi\,dz
\le
\int_{\Omega}u_{0}^{-\delta}\psi\,dz
\quad
\text{for every } \psi\in C_c^\infty(\Omega),\ \psi\ge0.
\]
This implies
\begin{equation}\label{eq:barrier-estimate}
\int_{\Omega}
\gradgr(u_{0}-\sigma)^+\cdot\gradgr(u-w)\,dz
\le
\int_{\Omega}u_{0}^{-\delta}(u-w)\,dz .
\end{equation}
We note that $u+u_{0}\ge\varepsilon$ on $\{u\ne w\}$ which implies $(u+u_{0})^{-\delta}(u-w)\in L^1(\Omega)$. Moreover, proceeding as in Lemma \ref{lem:weak_solution_identity} we have $u_{0}^{-\delta}(u-w)\in L^1(\Omega)$. Using \eqref{eq:weak-identity} and \eqref{eq:barrier-estimate}, we have
\begin{equation*}
\begin{aligned}
\int_{\Omega}
\big|\gradgr\big((u+(u_{0}-\sigma)^+-\varepsilon)^+\big)\big|^2&\,dz
=
\int_{\Omega}
\gradgr\big(u+(u_{0}-\sigma)^+-\varepsilon\big)\gradgr(u-w)\,dz \\
&\le
\int_{\Omega}u_{0}^{-\delta}(u-w)\,dz
-\int_{\Omega}g(z,u)(u-w)\,dz
+\int_{\Omega}v(u-w)\,dz \\
&=
\int_{\Omega}(u+u_{0})^{-\delta}(u-w)\,dz
+\int_{\Omega}v(u-w)\,dz \\
&\le
\varepsilon^{-\delta}\int_{\Omega}(u-w)\,dz
+\int_{\Omega}v(u-w)\,dz .
\end{aligned}
\end{equation*}
Therefore, for every $\varepsilon>0$ the family $\left\{ \big(u+(u_{0}-\sigma)^+ \varepsilon\big)^+:\sigma>0 \right\}$ is bounded in $\spacegr$. Passing to the limit as $\sigma\to0$ yields
\[
(u+u_{0}-\varepsilon)^+ \in \spacegr
\quad \text{for every } \varepsilon>0,
\]
which concludes the proof.
\end{proof}

\begin{lemma}
Let $\lambda>0$ and let $v\in \spacegr$ be positive such that
\begin{equation*}
-\Gr v = v^{-\delta} + \lambda v^p
\quad \text{in } \Omega
\end{equation*}
in the sense of distributions. Assume $(v-\varepsilon)^+ \in \spacegr$ for every $\varepsilon>0$. In the subcritical and critical cases, suppose $v\in L^{\frac{2pQ}{Q+2}}(\Omega)$, and in the supercritical case, suppose $v\in L^\infty(\Omega)$. Then $v-u_{0}$ is a positive weak solution of \eqref{eq:shifted_singular_problem} and $v-u_{0}\in L^\infty(\Omega)$.
\end{lemma}

\begin{proof}
Consider the auxiliary problem 
\begin{equation}\label{eq:auxillary_pb_fixed_rhs}
\begin{cases}
-\Gr u + g(z,u) = v^p & \text{in } \Omega,\\
u = 0 & \text{on } \partial\Omega .
\end{cases}
\end{equation}
It is straightforward to verify that $0$ is a strict subsolution of this problem. Define the functional $\tilde I:\spacegr\to(-\infty,\infty]$ by
\begin{equation*}
\tilde I(u)=
\begin{cases}
\displaystyle
\frac12\int_\Omega |\gradgr u|^2\,dz
+ \int_\Omega G(z,u)\,dz
- \lambda\int_\Omega v^p u\,dz,
& \text{if } G(z,u)\in L^1(\Omega),\\[1.2ex]
+\infty, & \text{otherwise}.
\end{cases}
\end{equation*}
Let $K_{0} = \left\{ u \in \spacegr : u \geq 0 \text{ a.e. in } \om \right\}$ and $\tilde{I}_{K_0}$ be the restriction of the functional $\tilde{I}$. Since the functional $\tilde{I}_{K_0}$ is bounded below, we can choose a minimizing sequence $\{u_k\}\subset K_0$ such that $\tilde I_{K_0}(u_k) \to \inf \tilde I_{ K_0}(K_0)$. Then there exists $u\in K_0$ such that $u_k \rightharpoonup u$ weakly and
\[
\tilde I_{K_0}(u)=\min \tilde I_{K_0}(K_0).
\]
Hence, $0\in \partial^- \tilde I_{K_0}(u)$. Since $0$ is a strict subsolution, Proposition \ref{prop:sub_super_solution_criterion} implies that $u$ is a nontrivial nonnegative weak solution of the problem \eqref{eq:auxillary_pb_fixed_rhs}.  
By Lemma~\ref{lem:truncation-property}, we have $(u+u_{0}-\varepsilon)^+ \in \spacegr$ for every $\varepsilon>0$. Moreover,
\begin{equation*}
\int_\Omega \big(\gradgr(u+u_{0})\cdot\gradgr w
-(u+u_{0})^{-\delta}w
-\lambda v^p w\big)\,dz=0
\end{equation*}
and
\begin{equation*}
\int_\Omega \big(\gradgr v\cdot\gradgr w
-v^{-\delta}w
-\lambda v^p w\big)\,dz=0
\end{equation*}
for every $w\in \spacegr\cap L^\infty(\Omega)$ with compact support. Applying Lemma \ref{lem:comparison} we deduce $v = u+u_{0}$. Consequently $v - u_{0}$ is a nonnegative solution of \eqref{eq:shifted_singular_problem}.

By the assumptions on $v$ and Lemma~\ref{lem:L-infty-regularity}, we conclude $u\in L^\infty(\Omega)$. Using interior regularity theory we further obtain $u\in C_{loc}^{0,\alpha}(\Omega)$. Finally,
\begin{equation*}
-\Gr u + g(z,u)
= \lambda (u+u_{0})^p >0
\quad \text{in }\Omega .
\end{equation*}
Since $g(z,u) \in L_{loc}^{1}(\om)$, by the strong maximum principle proved in Lemma \ref{lem:SMP_grushin} we conclude that $u>0$ in $\Omega$, completing the proof.
\end{proof}

\bibliographystyle{abbrv}
\bibliography{references}
\end{document}